\documentclass{article}

\usepackage{amsmath}
\usepackage{amssymb}
\usepackage{amsthm}
\usepackage{tikz-cd}
\usepackage{graphicx}

\tikzset{labl/.style={anchor=south, rotate=90, inner sep=.5mm}}
\tikzset{close/.style={near start,outer sep=-2pt}}

\DeclareMathOperator{\sol}{sol}
\DeclareMathOperator{\Dec}{Dec}

\DeclareMathOperator{\ab}{ab}

\DeclareMathOperator{\ur}{ur}
\DeclareMathOperator{\cd}{cd}

\DeclareMathOperator{\na}{na}
\DeclareMathOperator{\ch}{char}
\DeclareMathOperator{\Stab}{Stab}

\DeclareMathOperator{\tr}{tr}
\DeclareMathOperator{\ord}{ord}

\newcommand{\isom}{\cong}

\newcommand{\F}{\mathbb{F}}
\newcommand{\Lb}{\overline{L}} 
\newcommand{\tD}{\widetilde{\mathcal{D}}}

\newcommand{\D}{\mathcal{D}}

\newcommand{\G}{\mathcal{G}}

\newcommand{\Q}{\mathbb{Q}}
\newcommand{\Z}{\mathbb{Z}}
\newcommand{\p}{\mathfrak{p}}
\newcommand{\Pc}{\mathfrak{P}} 
\newcommand{\q}{\mathfrak{q}}

\newcommand{\K}{\overline{K}}

\newcommand{\Primes}{\mathfrak{Primes}^{\na}}
\newcommand{\lPrimes}{\mathfrak{Primes}^{\na,(l')}}

\newcommand{\surj}{\twoheadrightarrow}
\newcommand{\inj}{\hookrightarrow}
\newcommand{\isomto}{\xrightarrow{\sim}}

\newtheorem{theorem}{Theorem}
\numberwithin{theorem}{section}
\newtheorem{proposition}[theorem]{Proposition}
\newtheorem{corollary}[theorem]{Corollary}
\newtheorem{lemma}[theorem]{Lemma}

\newtheorem*{NU}{Neukirch-Uchida}
\newtheorem*{ST}{Sa\"idi-Tamagawa}
\newtheorem*{Uchida1}{Conditional Hom-Form}
\newtheorem*{Uchida2}{Hom-Form over $\Q$}
\newtheorem*{Theorem1}{Conditional $m$-step Hom-Form of the Birational Anabelian Geometry of Number Fields}
\newtheorem*{Theorem2}{$m$-step Hom-Form over $\Q$}
\newtheorem*{HF}{Hom-Form of the Birational Anabelian Geometry of Number Fields (Conjecture)}

\theoremstyle{definition}
\newtheorem{definition}[theorem]{Definition}

\theoremstyle{remark}
\newtheorem{remark}[theorem]{Remark}

\numberwithin{equation}{section}

\title{The $m$-step Solvable Hom-Form of Birational Anabelian Geometry for Number Fields}
\author{Alberto Corato\thanks{a.corato@exeter.ac.uk} \\ University of Exeter \and Mohamed Sa\"idi\thanks{m.saidi@exeter.ac.uk} \\ University of Exeter
\date{}}

\begin{document}

\maketitle

\thanks{This article was prepared with the financial support of EPSRC research grant EP/T031816/1.}

\begin{abstract}
In 1981, Uchida proved a \textbf{conditional} version of the Hom-form of the Grothendieck birational anabelian conjecture for number fields. In this paper we prove an $m$-step solvable \textbf{conditional} version of the Grothendieck birational anabelian conjecture for number fields whereby our conditions are slightly weaker than the ones in Uchida's theorem. Furthermore, as in Uchida's work, we show that our result is unconditional in the case where the number field relating to the domain of the given homomorphism is $\Q$. \end{abstract}

\maketitle

\section*{Notation and Definitions}
\begin{itemize}
\item A number field is a finite field extension of $\Q$.
\item A $p$-adic local field is a finite field extension of the field of $p$-adic numbers $\Q_p$.
\item Let $K$ be a number field. We denote by $\K$ a separable closure of $K$, and by $G_K = G(\K/K)$ its absolute Galois group.
\item Let $G$ be a profinite group. We denote $G[1]=\overline{[G,G]}$ the closure of its commutator subgroup, and $G^{\ab} = G/G[1]$.
\item Let $G$ be a profinite group. For all integers $m \ge 2$, we define inductively $G[m]=\overline{[G[m-1],G[m-1]]}$, and $G^m = G/G[m]$. In particular $G^{\ab} = G^1$.
\item Let $G$ be a profinite group, we denote $G^{\sol}=G/ (\bigcap_{m \ge 1} G[m])$ the maximal pro-solvable quotient of $G$.
\item Let $G$ be a profinite group and $p$ a prime number. We say that a subgroup $G_p$ of $G$ is a $p$-Sylow subgroup of $G$ if its order $|G_p|$ is a power of $p$, and $p \not\vert [G:G_p]$. Both $|G_p|$ and $[G:G_p]$ are considered as supernatural numbers.
\item Let $G$ be a profinite group, and $p$ a prime number. We denote by $G(p)$ the maximal pro-$p$-quotient of $G$, and by $G^{(p')}$ the maximal prime-to-$p$ quotient of $G$ 
\item If $K/k$ is a Galois extension of fields we denote its Galois group by $G(K/k)$. Given a profinite group $G$ we say that $K$ is a $G$-extension of $k$ if $G(K/k)$ is isomorphic to $G$.
\item Let $K$ be a number field, $p$ a prime number and $s \ge 1$ an integer. We say that $K$ has $\Z_p$-rank $s$ when $s$ is the maximum integer for which $K$ has a $\Z_p^s$-extension.
\item For a number field $K$, and a separable closure $\K$, we say that the subfield of $\K$ corresponding to $G_K[m]$ is the maximal $m$-step solvable extension of $K$, which we denote by $K_m$. Observe that $K_m$ is a Galois extension of $K$ with Galois group $G_K^m$. Furthermore, we set $K_0 = K$. \\
If $K=\Q$, we will denote this by $\Q_{(m)}$ instead of $\Q_m$ to avoid confusion with the $p$-adic numbers $\Q_p$ in case $m$ is a prime number. 
\item Let $K$ be an algebraic extension of $\Q$ (not necessarily finite). We denote by $\Primes_K$ the set of its non-archimedean primes, and for a prime number $l$ we denote by $\lPrimes_K$ the subset of $\Primes_K$ consisting of primes with residue characteristic $\neq l$.
\item For a field $k$, and a prime number $p$, we denote by $\mu(k)$ the (abelian) group of the roots of unity contained in $k$, by $\mu(k)(p)$ the subgroup of $\mu(k)$ of roots of $p$-power order, and by $\mu(k)^{(p')}$ the subgroup of prime-to-$p$-power order roots of unity of $\mu(k)$.
\item Let $K$ be an algebraic extension of $\Q$, and $\p \in \Primes_K$. We denote by $K_{\p}$ the localisation of $K$ at $\p$.
\item For a number field $K$, and $\p \in \Primes_K$ with residue characteristic $p$, we will denote by $f_{\p}$ its residue degree, by $e_{\p}$ its ramification index, and by $N(\p)$ its norm (i.e. $N(\p) = p^{f_{\p}}$ and $[K_{\p}:\Q_p]=e_{\p}f_{\p}$).
\item For the absolute Galois group $G_k$ associated to a local field $k$, we denote by $G_k^{\tr}$ its maximal tamely ramified quotient and by $G_k^{\ur}$ its maximal unramified quotient.
\item For a Galois extension $K/k$, and a prime $\p$ of $K$ (not necessarily non-archimedean), we denote the decomposition group at $\p$ in $G(K/k)$ by $D_{\p}$. If $\bar{\p} = \p \cap k$, we say $D_{\bar{\p}}=D_\p$ is a decomposition group above $\bar{\p}$ (unique up to conjugation by an element of $G(K/k)$).
\item For a Galois extension of $K/k$, and two primes $\p,\,\p'$ of $K$, we say that $\p$ and $\p'$ are conjugate in $G(K/k)$ if there exists $g \in G(K/k)$ such that $g \p = \p'$ where $g \p$ is the element of $\mathfrak{\Primes}_K$ determined by the natural action of $g$ on $\p$. 
\item Let $G$ be a profinite group, and $H_1,H_2 \subseteq G$ two subgroups. We say that $H_1$ and $H_2$ are commensurable if $H_1 \cap H_2$ is open in both $H_1$ and $H_2$.
\item We say that a group $F$ is $l$-decomposition-like if it is non-abelian and fits in an exact sequence $1 \to \Z_l \to F \to \Z_l \to 1$.
\item For a prime number $p$, we say that a group $G$ has cohomological $p$-dimension $n$ if $n$ is the largest integer for which $H^n(G,A)(p)=0$ for every torsion $G$-module $A$. We denote this $\cd_p(G)$. This is possibly $\infty$.
\item For a group $G$ we say it has cohomological dimension $\sup_p(\cd_p(G))$. We denote this number by $\cd(G)$. This is possibly $\infty$.
\item Let $K$ be a number field, and let $S \subseteq \Primes_{K}$. We define the density of the set $S$ in $\Primes_{K}$ as \[ \delta_K(S)= \lim_{s \to 1^+} \frac{\sum_{\p \in S} N(\p)^{-s}}{\sum_{\p \in \Primes_K} N(\p)^{-s}}. \]
\item All homomorphisms of profinite groups are assumed to be continuous with respect to the profinite topology
\end{itemize}

\section*{Introduction}

For two algebraic number fields $K$ and $L$, an isomorphism between their separable closure $\K$ and $\Lb$ induces an isomorphism between their absolute Galois groups $G_K=G(\K/K)$ and $G_L=G(\Lb/L)$. A well known result by Neukirch and Uchida (\cite{Neu1}, \cite{Neu2}, \cite{UchidaIsom}) shows that $G_K$ determines the isomorphy type of $K$ as follows:
\begin{NU}
Let $\sigma: G_K \to G_L$ be an isomorphism of profinite groups. Then, there exists a unique isomorphism of fields $\tau: \K \isomto \Lb$ such that for all $g \in G_K$ we have $\sigma(g) = \tau g \tau^{-1}$. Furthermore, $\tau$ restricts to an isomorphism $K \isomto L$.
\end{NU}
In \cite{UchidaIsom}, Uchida also shows that in the above theorem the absolute Galois group can be replaced by its maximal pro-solvable quotient $G^{\sol}$. \\ 
A more recent result, by Sa\"idi and Tamagawa \cite{S-T}, shows that in the statement of the Neukirch-Uchida theorem, one can replace $G_K$ by its $m$-step maximal solvable quotient $G_K^m$, as follows:
\begin{ST}
Let $m \ge 2$ be an integer, and let $\sigma_{m+3}: G_K^{m+3} \to G_L^{m+3}$ be an isomorphism of profinite groups. Consider the isomorphism of profinite groups $\sigma_m: G_K^m \to G_L^m$ naturally induced by $\sigma_{m+3}$. Then, there is a unique isomorphism of fields $\tau_m: K_m \to L_m$ such that for all $g \in G_K^m$ we have $\sigma_m(g) = \tau_m g \tau_m^{-1}$. Furthermore, $\tau_m$ restricts to an isomorphism $K \isomto L$.
\end{ST} 
Another noteworthy result by Uchida (\cite{UchidaHom}) shows that one can replace the isomorphisms in the Neukirch-Uchida theorem with homomorphisms. More precisely, one can obtain the following conditional result:
\begin{Uchida1}
Let $\sigma: G_K \to G_L$ be a homomorphism of profinite groups, and assume that for each decomposition group at a finite prime $D_{\p} \subset G_K$, there exists a decomposition group at a finite prime $D_{\q} \subset G_L$ such that $\sigma(D_{\p}) \subseteq D_{\q}$ and $\sigma(D_{\p})$ is open in $D_{\q}$. Then, $\sigma(G_K)$ is open in $G_L$, and there exists a unique embedding of fields $\tau: \Lb \inj \K$ such that for all $g \in G_K$ we have $\tau \sigma(g) = g \tau$. Furthermore, $\tau$ restricts to an embedding $L \inj K$.
\end{Uchida1}
Uchida also proved the following unconditional result:
\begin{Uchida2}
Let $\sigma: G_{\Q} \to G_L$ be a homomorphism of profinite groups such that $\sigma(G_{\Q})$ is open in $G_L$. Then, $L=\Q$, $\sigma$ is an isomorphism, and there exists an an automorphism $\tau: \overline{\Q} \to \overline{\Q}$ such that $\tau \sigma(g) = g \tau$ for all $g \in G_{\Q}^m$.
\end{Uchida2}
Uchida's results provide a conditional version of the Hom-Form of the Birational Anabelian Geometry of Number Fields (\cite{UchidaHom}, Conjecture)
\begin{HF}
Let $\sigma: G_K \to G_L$ be a homomorphism of profinite groups such that $\sigma(G_K)$ is open in $G_L$. Then, there exists a unique embedding of fields $\tau: \Lb \inj \K$ such that for all $g \in G_K$ we have $\tau \sigma(g) = g \tau$. Furthermore, $\tau$ restricts to an embedding $L \inj K$.
\end{HF} 
Our goal in this paper is to refine Uchida's results to a sharper version, analogous to the relation between Sa\"idi and Tamagawa's results on the Isom-Form with respect to the Neukirch-Uchida theorem. To do this, we will use the theory established by Sa\"idi and Tamagawa in \cite{S-T}, and prove such a version under some condition on a continuous homomorphism between $G_K^m$ and $G_L^m$ which can be detected group-theoretically. \\ 
We will use an object introduced by Sa\"idi and Tamagawa, $(\star_l)$-subgroups (see Definition \ref{starldef}), which consists of a generalisation of $l$-Sylow subgroups of decomposition groups which can be detected group-theoretically, but requiring an additional abelian ``step'' of information. We will show that, as long as our homomorphism satisfies some condition on these $(\star_l)$-subgroups, we are able to construct a map between the primes of the fields $K$ and $L$. By working on this map, we can show that it satisfies some inequality between the norm of primes. \\
With some additional result which we obtain from the map of primes, we are then able to show that a homomorphism $\sigma_m: G_K^m \to G_L^m$ satisfying our desired condition induces an embedding of $L_m$ in $K_m$ compatible with $\sigma_m$. \\
To show this embedding of fields is unique, we will add another condition, which can also be expressed entirely in terms of group theory, to our homomorphism. \\
The following statement (see Theorem \ref{Main} for the proof) will be the main result of this paper:
\begin{Theorem1}
Let $m \ge 2$ be an integer, and $\sigma_{m+3}: G_K^{m+3} \to G_L^{m+3}$ a homomorphism of profinite groups. Consider the homomorphism of profinite groups $\sigma_m: G_K^m \to G_L^m$ induced by $\sigma_{m+3}$. Assume that $\sigma_m$ restricts to an injection on every $(\star_l)$-subgroup of $G_K^{m}$ for every prime number $l$, and that $\sigma_{m+3}(G_K^{m+3}) \supseteq G_L^{m+3}[m-1]$. \\ Then, there exists a unique embedding of fields $\tau: L_m \inj K_m$ such that for all $g \in G_K^m$ we have $\tau \sigma_m(g) = g \tau$. Furthermore, $\tau$ restricts to an embedding $L \inj K$.
\end{Theorem1}  
The condition that $(\star_l)$-subgroups are mapped injectively is a sharper variation of the condition Uchida uses in his result (cf. Remark \ref{Sharp}).
\\ We will then also give a variation on the Conditional $m$-step Hom-Form which follows immediately from the above but is a proper sharpening of Uchida's result. In particular we will show that, for a possibly large integer $m$ which is dependent on the fields $K$ and $L$, the condition we require for uniqueness is automatically satisfied (see Theorem \ref{Variation}). \\
The last result we prove in this paper is a study of a specific case of our previous result, namely we look at what happens when $K = \Q$ and we have a homomorphism $\sigma_m: G_{\Q}^m \to G_L^m$ where $m$ is an integer $\ge 3$. We show that any homomorphism $\sigma_{m-2}: G_{\Q}^{m-2} \to G_L^{m-2}$ induced from $\sigma_m$ is an isomorphism and $L=\Q$. In particular, the condition on $l$-decomposition like group we requested in the main result holds. We then have the following unconditional result (see Theorem \ref{MainQ}):
\begin{Theorem2}
Let $m \ge 2$ be an integer, and $\sigma_{m+3}: G_\Q^{m+3} \to G_L^{m+3}$ a homomorphism of profinite groups. Consider the homomorphism of profinite groups $\sigma_m: G_{\Q}^m \to G_L^m$ induced by $\sigma_{m+3}$. Then, $L=\Q$, $\sigma_m$ is an isomorphism, and there exists a unique automorphism $\tau: \Q_{(m)} \to \Q_{(m)}$ such that, for all $g \in G_{\Q}^m$, we have $\sigma_m(g) = \tau g \tau^{-1}$.
\end{Theorem2}

\section{Structure of the maximal m-Step solvable Galois Group}

Most of the results in this section are found in Sa\"idi and Tamagawa's work \cite{S-T}. However, they are also presented here as they are needed to understand the proofs presented in the rest of the paper.

\begin{proposition}\label{mSeparatedness}
Let $m \ge 1$ be an integer, $K$ a number field, $\p_1,\p_2 \in \Primes_{K_m}$, and $D_{\p_1},D_{\p_2}$ their decomposition groups in $G_K^m$. Let $\bar{\p}_1$ and $\bar{\p}_2$ be respectively the images of $\p_1$ and $\p_2$ in $K_{m-1}$. Then $D_{\p_1} \cap D_{\p_2} \neq \{1\} \iff \bar{\p}_1 = \bar{\p}_2$. \\
Furthermore, if $m \ge 2$, $D_{\p_1} = D_{\p_2} \iff {\p}_1 = {\p}_2$. In particular the natural map $\Primes_{K_m} \surj \Dec(K_m/K)$ is bijective.
\end{proposition}
\begin{proof}
The proof for the first part of the statement can be found in \cite{S-T}, Proposition 1.3. The proof for the second part of the statement can be found in \cite{S-T}, Proposition 1.9
\end{proof}

\begin{definition}\label{curlyh}
For a subgroup $F$ of $G_K^m$, denote by $\widetilde{F}_0$ its inverse image by the quotient $G_K^{m+1} \to G_K^m$. We will denote by $\mathcal{H}^2(F,\F_l)$ the image of the inflation map $H^2(F,\F_l) \to H^2(\widetilde{F}_0,\F_l)$. \end{definition} 

\begin{proposition}\label{localglobal}
Let $F \subset G_K^m$ be a closed subgroup, and $\widetilde{F}$ be the inverse image of $F$ by the quotient $G_K \to G_L^m$. Let $L$ be the subextension of $K_m/K$ fixed by $F$. \\ Then, there exists an injective map \[\mathcal{H}^2(F,\F_l) \to \prod_{\p} H^2(\widetilde{F}_{\p},\F_l)\] where the product is indexed over all primes $\p$ in $\mathfrak{Primes}_L$, and $\widetilde{F}_{\p}$ denotes a decomposition subgroup at $\p$ in $\widetilde{F}$.
\end{proposition}
\begin{proof}
This is a restatement of Proposition 1.17 in \cite{S-T} in the more general case where $l$ is not necessarily odd and $L$ is not necessarily totally imaginary (cf. The proof of Proposition 1.14 in \cite{S-T}). 
\end{proof}

\begin{definition}\label{starldef}
Let $m \ge 2$ be an integer, $l$ a prime number, $F$ a closed subgroup of $G_K^m$. We will say that $F$ satisfies condition $(\star_l)$ (or, that $F$ is a $(\star_l)$-subgroup of $G_K^m$) if the following conditions hold:
\begin{itemize}
\item There is an exact sequence $1 \to U_1 \to F \to U_2 \to 1$, where $U_1$ and $U_2$ are isomorphic to $\Z_l$.
\item $\mathcal{H}^2(F,\F_l)$ is non trivial.
\end{itemize}
We also define the set $\tD_{m,l,K}$ as the set of all closed subgroups of $G_K^m$ satisfying condition $(\star_l)$. 
\end{definition}

\begin{definition}\label{equivdef}
Let $m \ge 2$ be an integer, $l$ a prime number. For $F_1,F_2 \in \tD_{m,l,K}$, we say that $F_1 \approx F_2$ if and only if for any open subgroup $H$ of $G_K^m$ containing $G_K^m[m-1]$, the images of the subgroups $F_1 \cap H$ and $F_2 \cap H$ by the map $H \surj H^{\ab}$ are commensurable. \\
The relation $\approx$ is an equivalence relation on $\tD_{m,l,K}$, and we denote the set of equivalence classes $\tD_{m,l,K}/\approx$ by $\D_{m,l,K}$. \end{definition}
 
\begin{proposition}\label{starlisopen}
Let $m \ge 2$ be an integer, $l$ a prime number. Then, $F$ satisfies condition $(\star_l)$ if and only if there exists a prime $\p \in \Primes_{K_m}$ with $\ch(\p) \neq l$ such that $F$ is an open subgroup of an $l$-Sylow subgroup of the decomposition group $D_{\p} \subset G_K^m$ of $\p$. Such a prime $\p$ need not be unique, however the restriction $\bar{\p} \in \Primes_{K_{m-1}}$ of $\p$ to $K_{m-1}$ is uniquely determined. \\
In particular, a surjective map $\widetilde{\phi}: \tD_{m,l,K} \surj \lPrimes_{K_{m-1}}$ is well-defined. 
\end{proposition}
\begin{proof}
This is a restatement of \cite{S-T}, Proposition 1.22.
\end{proof}

We may compose $\widetilde{\phi}$ with the natural map $\lPrimes_{K_{m-1}} \to \Primes_{K_{m-1}}$, so to obtain a map $\tD_{m,l,K} \to \Primes_{K_{m-1}}$ which is not surjective.  

\begin{proposition}\label{starldown}
Let $m \ge 2$ be an integer, $l$ a prime number, and $F \subset G_K^{m+1}$ a $(\star_l)$-subgroup. Consider the image $\overline{F}$ of $F$ by the quotient $G_K^{m+1} \to G_K^m$. Then, $\overline{F}$ satisfies property $(\star_l)$ (i.e. $\overline{F} \in \tD_{m,l,K}$).
\end{proposition}
\begin{proof}
Let $\p \in \Primes_{K_{m+1}}$ be a prime with decomposition group $D_{\p} \subset G_K^{m+1}$ such that $F \subseteq D_{\p}$ (cf. Proposition \ref{starlisopen}), and let $\bar{\p}$ be its image in $\Primes_{K_m}$. Then, $\ker(D_\p \surj D_{\bar{\p}}) \cap F$ is trivial (cf. \cite{S-T}, Proposition $1.1.(vii)$) which means the quotient map restricts to an isomorphism on $F$, and $\overline{F}$ is an $l$-decomposition like group. However, the image of $D_{\p}$ in $G_K^m$ by the quotient map $G_K^{m+1} \surj G_K^m$ coincides with the decomposition group $D_{\bar{\p}}$ at $\bar{\p}$. Since $\overline{F} \subset D_{\bar{\p}}$ and $\cd_l (F)=2$, $\overline{F}$ is $l$-open in an $l$-Sylow subgroup of $D_{\bar{\p}}$, and by Proposition \ref{starlisopen} this shows $\overline{F}$ satisfies condition $(\star_l)$.
\end{proof}

\begin{proposition}\label{STMapping}
Let $m \ge 2$ be an integer, and $l$ a prime number. Then, the following hold:
\begin{enumerate}
\item The map $\widetilde{\phi}_{m,l,K}: \tD_{m,l,K} \to \lPrimes_{K_{m-1}}$ defined as in Proposition \ref{starlisopen} factors through $\D_{m,l,K}$. In particular, there is a naturally defined map $\phi_{m,l,K}: \D_{m,l,K} \to \Primes_{K_{m-1}}$ which is injective.
\item The map $\phi_{m,l,K}$ restricts to a bijection $\D_{m,l,K} \isomto \lPrimes_{K_{m-1}}$.
\item The map $\phi_{m,l,K}$ is $G_K^m$-equivariant with respect to the natural actions of $G_K^m$ on $\D_{m,l,K}$ and $\Primes_{K_{m-1}}$. In particular, the action of $G_K^m$ on $\D_{m,l,K}$ factors through $G_K^{m-1}$.
\item Let $a \in \D_{m,l,K}$, and let $\bar{\p} = \phi_{m,l,K}(a)$. Then the stabiliser $\Stab(a)=\{g \in G_K^{m-1}\,|\,g a = a\}$ coincides with the decomposition group $D_{\bar{\p}}$ of $\bar{\p}$ in $G_K^{m-1}$. 
\end{enumerate}
\end{proposition}
\begin{proof}
This is a restatement of \cite{S-T}, Proposition 1.23.
\end{proof}

The authors would like to thank Akio Tamagawa for the following result and its proof, which allows us to restrict the order of the torsion elements in $G_K^m$ and relate them to decomposition groups at archimedean primes.

\begin{proposition}\label{TamagawaTorsion}
Let $m \ge 2$ be an integer, and let $g \in G_K^m$ be a (non-trivial) torsion element. Then, $g$ is of order $2$, and the image $\bar{g}$ of $g$ by the projection $G_K^m \surj G_K^{m-1}$ is non-trivial. Furthermore, the subgroup $\langle \bar{g} \rangle$ of $G_K^{m-1}$ generated by $\bar{g}$ is the decomposition group of an archimedean prime of $K^{m-1}$.
\end{proposition}
\begin{proof}
Let $\widehat{\Q}$ be the unique $\widehat{\Z}$-extension of $\Q$, and let $\widehat{K}$ be the composite extension $K \widehat{\Q}$ of $\Q$. By construction, $G(\widehat{K}/K)$ is given by an open subgroup of $G(\widehat{\Q}/\Q) \isom \widehat{\Z}$, and so isomorphic to $\widehat{\Z}$ itself. Observe that $\widehat{K}$ is an abelian extension of $K$. \\ By (\cite{SerreGC}, II, Proposition 9), as $p^{\infty} \, |\, [\widehat{K}:K]$ for all prime numbers $p$, we have $\cd_p G_{\widehat{K}} \le 1$ and in particular $\cd G_{\widehat{K}} \le 1$ (cf. Corollary 8.1.18 in \cite{NWS}). Then $G_{\widehat{K}}^{\ab}$ is a free abelian group, and in particular it is torsion-free. The same argument also holds true for any separable extension of $\widehat{K}$. \\ 
Let us take a torsion element $g \in G_K^m$ and denote by $n$ its order ($g$ is possibly trivial). We first show the former part of the statement, namely that if $g$ is non-trivial it is of order $2$. We will start assuming $m=2$, and then show inductively the case $m > 2$. \\
Let us then assume $m=2$, and also $K$ is totally imaginary. As $\widehat{\Q}$ is necessarily totally real, $K$ and $\widehat{\Q}$ are disjoint extension of $\Q$ so $G(\widehat{K}/K) \isom \widehat{\Z}$. Consider the absolute Galois groups $G_K$ and $G_{\widehat{K}}$, and the maximal abelian extension $\widehat{K}^{\ab}$ of $\widehat{K}$. This extension is, by construction, a metabelian Galois extension of $K$, therefore there exists a quotient $G_K^2 \surj G(\widehat{K}^{\ab}/K)$. Furthermore, since $\widehat{K}/K$ is abelian, it is a subextension of $K^{\ab}/K$, and as the extension $K^{\ab}/\widehat{K}$ is necessarily also abelian, it is a subextension of $\widehat{K}^{\ab}/\widehat{K}$. We then have a commutative diagram \begin{center}
\begin{tikzcd} 
1 \arrow[r] & G_K^2[1] \arrow[r] \arrow[d] & G_K^2 \arrow[r] \arrow[d, two heads] & G_K^{\ab} \arrow[r] \arrow[d, two heads] & 1 \\
1 \arrow[r] & G_{\widehat{K}}^{\ab} \arrow[r] & G(\widehat{K}^{\ab}/K) \arrow[r] \arrow[ur,two heads] & G(\widehat{K}/K) \arrow[r] & 1
\end{tikzcd}
\end{center}
where $G(\widehat{K}/K)$ is torsion-free by construction ($\isom \widehat{\Z}$), and $G_{\widehat{K}}^{\ab}$ and $G_K^2[1]=G_{K^{\ab}}^{\ab}$ are torsion-free by the argument at the start of the proof. \\ It follows that $G(\widehat{K}^{\ab}/K)$ is torsion-free. We then have the image of $g$ in the quotient $G_K^2 \surj G(\widehat{K}^{\ab}/K)$ is trivial, and by commutativity in the diagram it follows that $g \in G^2_K[1]$. However $G_K^2[1] = G_{K^{\ab}}^{\ab}$ is torsion-free as well, which shows $g$ can only be trivial. \\
Let us assume now $K$ is not totally imaginary. The composite $\widehat{K}/K$ is still a Galois extension with Galois group $\hat{\Z}$, and the extensions $\widehat{K}$ and $K(i)$ are disjoint over $K$, and we consider the composite $\widehat{K}' = \widehat{K}\cdot K(i)$. Then, the extension $\widehat{K}'/K$ has Galois group $\widehat{\Z} \times \Z/2\Z$, and in particular it is an abelian extension of $K$. Replacing $\widehat{K}$ with $\widehat{K}'$ in the argument above yields a commutative diagram
\begin{center}
\begin{tikzcd} 
1 \arrow[r] & G_K^2[1] \arrow[r] \arrow[d] & G_K^2 \arrow[r] \arrow[d, two heads] & G_K^{\ab} \arrow[r] \arrow[d, two heads] & 1 \\
1 \arrow[r] & G_{\widehat{K}'}^{\ab} \arrow[r] & G({\widehat{K}'}\hspace{0pt}^{\ab}/K) \arrow[r] \arrow[ur,two heads] & G(\widehat{K}'/K) \arrow[r] & 1
\end{tikzcd}
\end{center}
and either $g$ is trivial or $g \not\in G_K^2[1]$ (as above), which means $g$ maps to a non-trivial element in $G_K^{\ab}$ of order $n$. As this map factor through $G_K^2 \surj G({{\widehat{K}'}}\hspace{0pt}^{\ab}/K)$, denote by $\hat{g}$ the image of $g$ through this map (which needs to also be of order $n$) and as $\hat{g} \not\in G_{\widehat{K}'}^{\ab}$, it is mapped to an element in a torsion subgroup of $G(\widehat{K}'/K)$. However, this implies that $\ord(g) \le 2$ so either $g$ is again trivial, or $g$ is of order $2$. \\
For $m > 2$, we may use the exact sequence \[ 1 \to G_K^m[m-1] \to G_K^m \to G_K^{m-1} \to 1 \] together with the fact that $G_K^m[m-1] = G_{K_m}^{\ab}$ is torsion-free to show that any non-trivial torsion element $g \in G_K^m$ is mapped to a non-trivial torsion element in $G_K^{m-1}$ of the same order, and so any non-trivial torsion element in $G_K^m$ is of order $2$ as well. \\
We now need to prove that the image $\bar{g}$ of $g$ by the quotient map $G_K^m \surj G_K^{m-1}$ corresponds to complex conjugation. Assume then that $g$ is an element of order $2$, and let $T=\langle g \rangle $ which is isomorphic to $\Z / 2\Z$. Let $\overline{T} = \langle \bar{g} \rangle$ be the image of $T$ by the quotient map $G_K^m \surj G_K^{m-1}$. Naturally, the inverse image of $\overline{T}$ by this quotient must contain the torsion group $T$, so the group $\mathcal{H}^2(\overline{T},\F_2)$ (cf. Definition \ref{curlyh}) will be non-trivial. \\
Let $\widetilde{T}$ be the inverse image of $\overline{T}$ by the quotient $G_K \surj G_K^{m-1}$, and let $L$ be the subextension of $K_{m-1}/K$ fixed by $\overline{T}$. By Proposition \ref{localglobal}, we have an injective map \[ \mathcal{H}^2(\overline{T},\F_2) \to \prod_{\p \in \mathfrak{Primes}_L} H^2(\widetilde{T}_\p,\F_2) \] and so there must exist a prime $\p$ for which the image in $H^2(\widetilde{T}_\p,\F_2)$ is non-trivial, which implies $\widetilde{T}_{\p}$ is non-trivial. \\ We may then take the image $\overline{T}_{\p}$ of $\widetilde{T}_{\p}$ in $G_K^{m-1}$, which must necessarily contain $\bar{g}$ and must be contained in some decomposition group $D_{\p} \subset G_K^{m-1}$ (by abuse of notation, we denote both the prime in $L$ and its image in $K$ by $\p$).\\
If $m-1 \ge 2$, $\bar{g} \in \overline{T}_{\p}$ implies $\p$ is necessarily archimedean as  decomposition groups of non-archimedean primes in $G_K^{m-1}$ are torsion-free (\cite{S-T}, Proposition 1.1.$(x)$) . \\
If $m-1=1$, that is $m=2$, then $\bar{g}$ needs to be contained in $I_{\p}$ as it is mapped trivially by the quotient $D_{\p} \surj D_{\p}^{\ur} \isom \widehat{\Z}$. The inverse image of $\overline{T}$ in the decomposition group $G_{\p} \subset G_K$ then also has to be contained in the inertia subgroup of $G_{\p}$, however for all non-archimedean primes $H^2(I_{\p},\F_2)=0$, which implies $\p$ is an archimedean prime as desired.
\end{proof}

\section{Homomorphisms of m-step solvable Profinite Groups}

In this section, we look at a few basic properties of homomorphisms of maximal $m$-step solvable Galois groups, and how they restrict to $(\star_l)$-subgroups.
\begin{lemma}\label{DiagramLemma}
Let $m \ge 1$ be an integer, $G$ and $H$ be profinite groups and $\sigma: G \to H$ a homomorphism of profinite groups. Then, there is a unique homomorphism of profinite groups $\bar{\sigma}: G^m \to H^m$ so that we have a commutative diagram \begin{center} \begin{tikzcd}
G \arrow["\sigma",r] \arrow[two heads,d] & H \arrow[two heads,d] \\
G^m \arrow["\bar\sigma",r] & H^m 
\end{tikzcd}
\end{center}
\end{lemma}
\begin{proof}
Naturally, we have $\sigma(G[m]) = \sigma(G)[m] \subseteq H[m]$. Then, if we consider the composite $G \to H \surj H^m$, we may observe that $G[m]$ is contained in the kernel of the composite. The statement then follows.
\end{proof}
For the rest of the section, we will denote by $K$ and $L$ two number fields, and we will consider their maximal $m$-step solvable Galois groups $G_K^m$ and $G_L^m$.
\begin{proposition}\label{Diagram}
Let $m \ge 1$ be an integer, and $\sigma_m: G_K^m \to G_L^m$ a homomorphism of profinite groups. Then, there exists a unique homomorphism of profinite groups $\sigma_{m-1}: G_K^{m-1} \to G_L^{m-1}$ such that the following diagram, where the vertical arrows are the canonical projection, is commutative:
\begin{center}
\begin{tikzcd}
G_K^{m} \arrow[r,"\sigma_{m}"] \arrow[two heads,d] & G_L^{m} \arrow[two heads,d] \\
G_K^{m-1} \arrow[r,"\sigma_{m-1}"] & G_L^{m-1}
\end{tikzcd}
\end{center}
Furthermore, if $\sigma_m(G_K^m)$ is open in $G_L^m$, then $\sigma_{m-1}(G_K^{m-1})$ is open in $G_L^{m-1}$.
\end{proposition}
\begin{proof}
The first part of the statement follows from Lemma \ref{DiagramLemma}. \\ 
Let $\widetilde{L}$ be the finite extension of $L$ corresponding to the image of $\sigma_m$. Then, it follows that the image of the composite $G_K^m \to G_L^m \surj G_L^{m-1}$ corresponds to the field $\widetilde{L} \cap L_{m-1}$ which is finite over $L$. It follows then that the image of $\sigma_{m-1}$ is open by commutativity of the diagram. \end{proof}

\begin{remark}\label{quotients}
With the notation of Proposition \ref{Diagram}, we will say that $\sigma_{m-1}$ is induced from $\sigma_m$. \\
Furthermore, if $L'$ is an extension of $L$ contained in $L_m$, and $K'$ is the field corresponding to $\sigma^{-1}_m (G(L_m/L'))$, we will say that $L'$ corresponds to $K'$ by $\sigma_m$. If $L'/L$ is a Galois extension, then $K'/K$ is also Galois, and in particular by taking the quotients with the respective Galois groups, we have an injective homomorphism $\sigma: G(K'/K) \to G(L'/L)$ naturally induced from $\sigma_m$. When $\sigma_m$ is surjective, this map $\sigma$ is an isomorphism.\end{remark}

We want to study the possible images of $(\star_l)$-subgroups of $G_K^m$ by a homomorphism $\sigma_m: G_K^m \to G_L^m$. It is useful to recall first what the possible images of an $l$-decomposition-like group $F \subset G_1$ with respect to a homomorphism of profinite groups $\sigma: G_1 \to G_2$ are (i.e. the possible quotients of $F$). These images are characterised as in the following classification, similar to the one presented by Uchida in (\cite{UchidaHom}, Section 1). \\
We let $U$ be the unique maximal abelian normal subgroup of $F$ (which we know is isomorphic to $\Z_l$), and $V = F/U$. By construction, $V$ is also isomorphic to $\Z_l$. We also set $H = F \cap \ker(\sigma)$.
Then, exactly one of the following is true for $\sigma(F)$:
\begin{enumerate}
\item\label{declikeinj} $\sigma(F)$ fits in an exact sequence $1 \to \Z_l \to \sigma(F) \to \Z_l \to 1$. In this case, $\sigma$ restricts to an injective map on $F$, that is $F $ is isomorphic to $\sigma(F)$.
\item $\cd(\sigma(F)) \le 1$. In this case either of the following holds:\begin{enumerate}
\item $H=F$, that is $\sigma(F)$ is trivial.
\item $H=U$, and $\sigma(F)$ is isomorphic to $\Z_l$.
\end{enumerate}
\item\label{ClassTorsion} $\sigma(F)$ contains torsion elements, and so $\cd(\sigma(F))=\infty$. In particular, we are in exactly one of the following cases:
\begin{enumerate}
\item $H \supset U$, and the map $F \surj \sigma(F)$ factors through $V$. In particular, $\sigma(F)$ is finite, non-trivial and cyclic.
\item $H$ is a non trivial proper open subgroup of $U$, and there is an exact sequence $1 \to E \to \sigma(F) \to \Z_l \to 1$, where $E$ is a non-trivial finite cyclic $l$-group.
\item $H$ is a proper open subgroup of $F$ and $H \cap U$ is a proper subgroup of $U$. Then, $\sigma(F)$ is metabelian and for $E_1,E_2$ non-trivial finite cyclic $l$-groups $\sigma(F)$ fits in an exact sequence $1 \to E_1 \to \sigma(F) \to E_2 \to 1$.
\end{enumerate}
\end{enumerate}
Thanks to Proposition \ref{TamagawaTorsion}, if we have a homomorphism $\sigma_m: G_K^m \to G_L^m$ for $m \ge 2$, we can exclude that case \ref{ClassTorsion} happens, unless $l=2$. \\ 
We also have the following consequence of case \ref{declikeinj} in the above classification:
\begin{lemma}\label{zlinj}
Let $F$ and $F'$ be $l$-decomposition like groups, and $\sigma: F \to F'$ a surjective homomorphism of profinite groups. Then, $\sigma$ is an isomorphism.
\end{lemma}
When we are working with $(\star_l)$-subgroups, the above Lemma specifies to the following result:
\begin{proposition}\label{MapStarl}
Let $m \ge 2$ be an integer $\sigma_{m+1}: G_K^{m+1} \to G_L^{m+1}$ a homomorphism of profinite groups, and consider the homomorphism $\sigma_m: G_K^m \to G_L^m$ induced from $\sigma_{m+1}$. \\
Let $F \in \tD_{m,l,K}$ be a $(\star_l)$-subgroup of $G_K^m$, and assume $\sigma_m$ restricts to an injection on $F$. Then, $F' =\sigma_m(F) \in \tD_{m,l,L}$, that is $F'$ satisfies condition $(\star_l)$.
\end{proposition}
\begin{proof}
As $\sigma_m$ restricts to an injective map on $F$, it induces naturally an isomorphism $F \isomto F'$. We then immediately get that $F'$ fits in an exact sequence \[ 1 \to \Z_l \to F' \to \Z_l \to 1. \]
To show that $F'$ satisfies condition $(\star_l)$ it only remains to be shown that the inflation map in Definition \ref{starldef} has non-trivial image. \\ Denote by $\widetilde{F}$ (resp. $\widetilde{F}'$) the inverse image of $F$ in $G_K^{m+1}$ with respect to the canonical projection (resp. of $F'$ in $G_L^{m+1}$). By commutativity of the diagram in Proposition \ref{Diagram}, it follows that $\widetilde{F}' = \sigma_{m+1}(\widetilde{F})$. Then, we may get the following commutative diagram
\begin{center}
\begin{tikzcd}
 H^2(\widetilde{F}',\F_l) \arrow[r] & H^2(\widetilde{F},\F_l)\\
 H^2(F',\F_l) \arrow[r,"\sim"] \arrow[u,"\inf_{F'}"] & H^2(F,\F_l)  \arrow[u,"\inf_F"] 
\end{tikzcd}
\end{center}
where the vertical arrows are the inflation maps and the horizontal arrows are induced by $\sigma_{m+1}$ and $\sigma_m$. Since $\sigma_m$ is an isomorphism between $F$ and $F'$, the induced arrow $H^2(F',\F_l) \to H^2(F,\F_l)$ is also an isomorphism. \\
By Definition \ref{starldef} the map $\inf_F: H^2(F,\F_l) \to H^2(\widetilde{F},\F_l)$ has non-trivial image, and so the composite map $H^2(F',\F_l) \to H^2(\widetilde{F},\F_l)$ will also have non-trivial image. By commutativity of the diagram the map $\inf_{F'}: H^2(F',\F_l) \to H^2(\widetilde{F}',\F_l)$ also has non-trivial image. We may now conclude the subgroup $F'$ of $G_L^m$ also satisfies property $(\star_l)$. 
\end{proof}

\begin{proposition}\label{starllift}
Let $m \ge 2$ be an integer. Consider a homomorphism of profinite groups $\sigma_{m+2}: G_K^{m+2} \to G_L^{m+2}$, and the homomorphisms $\sigma_{m+1}$ and $\sigma_m$ it induces.
Let $F \subset G_K^{m+1}$ satisfy property $(\star_l)$, and let $\overline{F}$ be the image of $F$ in $G_K^m$ (which also satisfies property $(\star_l)$, see Proposition \ref{starldown}).\\
Then the restriction of $\sigma_m$ to $\overline{F}$ is injective if and only if the restriction of $\sigma_{m+1}$ to $F$ is.
\end{proposition}
\begin{proof}
First we prove the ``if''. As $\sigma_{m+1}(F)$ satisfies condition $(\star_l)$ by Proposition \ref{MapStarl}, its image in $G_L^m$ must also satisfy condition $(\star_l)$ by Proposition \ref{starldown}. Then, by the commutativity of the diagram in Proposition \ref{Diagram} it follows the image of $\overline{F}$ by $\sigma_m$ is an $l$-decomposition like group, and so we conclude by Lemma \ref{zlinj}. \\
Now, we prove the ``only if'' part. This follows again from the commutativity of the diagram in Proposition \ref{Diagram}, as the composite of $F \isomto \overline{F} \isomto \sigma_m(\overline{F})$ must factor through $\sigma_{m+1}$. However, since the composite map $F \to \sigma_m(\overline{F})$ is an isomorphism, it follows that $\sigma_{m+1}$ restricts to an injection on $F$ as desired.
\end{proof}

\begin{proposition}\label{allclass}
Let $m \ge 2$ be an integer, and let $\sigma_m: G_K^m \to G_L^m$ be a homomorphism of profinite groups.
Let $F,F' \in \tD_{m,l,K}$ such that $\phi_{m,l,K}(F)$ and $\phi_{m,l,K}(F')$ are above the same prime $\p$ of $K$. Then, $\sigma_m$ restricts to an injection on $F$ if and only if it restricts to an injection on $F'$.
\end{proposition}
\begin{proof}
For some decomposition groups $D_\p, D'_{\p} \subset G_K^m$ above $\p$ and some appropriate $l$-sylow subgroups, we have $F \subseteq D_{\p,l}$ and $F' \subseteq D'_{\p,l}$. \\
As $D_{\p}$ and $D'_{\p}$ are conjugate in $G_K^m$, their $l$-Sylow subgroups also are, and in particular there exists $g \in G_K^m$ such that $g D_{\p,l} g^{-1} = D'_{\p,l}$. It then also follows that the images of $D_{\p,l}$ and $D'_{\p,l}$ by $\sigma_m$ are conjugate by $\sigma_m(g)$. \\
Assume without loss of generality that $F$ is mapped injectively. Then, $\sigma_m(D_{\p,l})$ must contain an $l$-decomposition like group, and by the classification $\sigma_m$ restricts to an injection on $D_{\p,l}$. It follows that $\sigma_m$ also restricts to an injection on $D'_{\p,l}$ and, in particular, on $F'$.
\end{proof}

\begin{remark}\label{LocalInj}
With the above proposition, for a fixed prime number $l$, we may say that whether $\sigma_m$ restricts to an injection on a subgroup $F \subseteq G_K^m$ satisfying property $(\star_l)$ or not only depends on the unique $\p$ of $K$ determined by $F$. \\ Together with Proposition \ref{starllift}, we can see that, as long as we have the conditions to recover $(\star_l)$-subgroups this is also independent of $m$, and really only depends on $\p$. \end{remark}

\begin{proposition}\label{MapApprox}
Let $m \ge 2$ be an integer, $l$ a prime number, and consider a homomorphism of profinite groups $\sigma_{m+1}: G_K^{m+1} \to G_L^{m+1}$. Let $F_1$ and $F_2$ be subgroups of $G_K^m$ satisfying condition $(\star_l)$ such that $\sigma_m: G_K^m \to G_L^m$ restricts to an injection on both $F_1$ and $F_2$, and $F_1 \approx F_2$ in $\tD_{m,l,K}$. \\
Then, $\sigma_m(F_1) \approx \sigma_m(F_2)$ in $\tD_{m,l,L}$. 
\end{proposition}
\begin{proof}
Let $H'$ be a subgroup of $G_L^{m}$ containing $G_L^m[m-1]$, and let $H=\sigma_m^{-1}(H')$. As $\sigma_m(G_K^m[m-1]) \subseteq G_L^m[m-1]$, it follows immediately that $H$ contains $G_K^m[m-1]$. \\
For $i=1,2$, let $F'_i=\sigma_m(F_i)$, let $\overline{F}_i$ be the image of $F_i \cap H$ in the quotient $H \to H^{\ab}$, and similarly let $\overline{F}'_i$ be the image of $F'_i \cap H'$ in $H' \to H'^{\ab}$. It follows from the definition of $\approx$ that $\overline{F}_1$ and $\overline{F}_2$ are commensurable. The goal is to show $\overline{F}'_1$ and $\overline{F}'_2$ are also commensurable. \\
As $\sigma_m$ restricts to an injection on $F_i$ and $\sigma_m(H) \subseteq H'$, it follows that $\sigma_m$ restricts to an isomorphism between $F_i \cap H$ and $F'_i \cap H'$. Then, this also implies we have an isomorphism $\overline{F}_i \to \overline{F}'_i$. In particular, these isomorphisms give the intersection $\overline{F}'_1 \cap \overline{F}'_2$ is an open subgroup of both $\overline{F}'_1$ and $\overline{F}'_2$, which shows they are indeed commensurable, as desired.
\end{proof} 

\section{Construction of a mapping of Primes}

In this section we maintain the same notation as in the previous section, so $K$ and $L$ are number fields. \\
With the results of Section 2, it is possible to use Proposition \ref{STMapping} to establish a map of primes for homomorphisms of $m$-step solvably closed Galois groups 

\begin{definition}\label{DaggerlDef}
For an integer $m \ge 2$ and a prime number $l$, we will say that a homomorphism of profinite groups $\sigma_m: G_K^m \to G_L^m$ satisfies condition $(\dagger_l)$ if the image by $\sigma_m$ of every subgroup of $G_K^m$ satisfying property $(\star_l)$ is a subgroup of $G_L^m$ satisfying property $(\star_l)$. \\
We will also say that $\sigma_m$ satisfies condition $(\dagger)$ if it satisfies condition $(\dagger_l)$ for all prime numbers $l$
\end{definition}

\begin{remark}\label{localinj2}
As in Remark \ref{LocalInj}, we know that we can detect whether a subgroup $F$ of $G_K^m$ satisfying condition $(\star_l)$ is mapped injectively by $\sigma_m$ or not by looking at any other $(\star_l)$-subgroup of $G_K^m$ which determines the same prime of $K$ as $F$ does. Furthermore, this can be also determined by looking at any other (suitable) number $m$ of steps. 
\end{remark}

\begin{remark}\label{groupdaggerl}
Let us assume that $\sigma_m$ is induced by a homomorphism of profinite groups $\sigma_{m+1}: G_K^{m+1} \to G_L^{m+1}$, so that we can recover group theoretically the subgroups of $G_K^m$ satisfying condition $(\star_l)$. Then, we may reformulate Definition \ref{DaggerlDef} using the following group theoretic characterization: $\sigma_m$ satisfies condition $(\dagger_l)$ if $\sigma_m$ restricts to an injection on all the subgroups of $G_K^m$ satisfying property $(\star_l)$ (cf. Proposition \ref{MapStarl})
\end{remark}

\begin{proposition}\label{daggerldown}
For an integer $m \ge 2$ and a prime number $l$, assume the homomorphism $\sigma_{m+1}: G_K^{m+1} \to G_L^{m+1}$ satisfies property $(\dagger_l)$. Then, the induced homomorphism $\sigma_m: G_K^m \to G_L^m$ also satisfies property $(\dagger_l)$.
\end{proposition}
\begin{proof}
Let $F$ be a subgroup of $G_K^m$ satisfying property $(\star_l)$. Then, for some prime $\p \in \Primes_{K_m}$ with decomposition group $D_{\p} \subseteq G_K^m$ we have $F$ is an open subgroup of $D_{\p,l}$ (cf. Proposition \ref{starlisopen}). For a prime $\hat{\p}$ of $\Primes_{K_{m+1}}$ with decomposition group $D_{\hat{\p}} \subset G_K^{m+1}$ we have a natural surjective map $D_{\hat{\p},l} \surj D_{\p,l}$, which is an isomorphism by Lemma \ref{zlinj}. Then, an open subgroup $H$ of $D_{\hat{\p},l}$ maps isomorphically to $F$. It follows necessarily that $H$ satisfies property $(\star_l)$, and is mapped injectively by $\sigma_{m+1}$ by assumption. By Proposition \ref{starllift}, it then follows $F$ is mapped injectively by $\sigma_{m}$, that is to a $(\star_l)$-subgroup of $G_L^m$ by Proposition \ref{MapStarl}. It then follows immediately that $\sigma_m$ does indeed verify condition $(\dagger_l)$.
\end{proof}

\begin{proposition}\label{DaggerlMap}
Let $m \ge 2$ be an integer, $l$ a prime number and $\sigma_{m}: G_K^{m} \to G_L^{m}$ a homomorphism of profinite groups that satisfies condition $(\dagger_l)$. Then, there is a natural map $\widetilde{\sigma}^*_{m,l}: \tD_{m,l,K} \to \tD_{m,l,L}$ defined by setting $\widetilde{\sigma}^*_{m,l}(F) = \sigma_{m,l}(F)$.
Furthermore, $\widetilde{\sigma}^*_{m,l}$ induces a map $\sigma^*_{m,l}: \D_{m,l,K} \to \D_{m,l,L}$.
\end{proposition}
\begin{proof}
The existence of $\widetilde{\sigma}^*_{m,l}$ follows immediately from Definition \ref{DaggerlDef}. \\
The existence of $\sigma^*_{m,l}$ then follows from \ref{MapApprox}.
\end{proof}

\begin{remark}\label{lMapPrimes}
With the assumptions of Proposition \ref{DaggerlMap} we have a naturally defined map \[\theta_{m-1}^{(l')}: \lPrimes_{K_{m-1}} \to \lPrimes_{L_{m-1}} \] which we obtain from $\sigma^*_{m,l}$ by considering the bijections $\phi_{m,l,K}$ and $\phi_{m,l,L}$ derived from Proposition \ref{STMapping}. In particular for each prime number $l$ for which $\sigma_m$ satisfies condition $(\dagger_l)$ we have the following commutative diagram
\begin{center}
\begin{tikzcd}
\tD_{m,l,K} \arrow["\widetilde{\sigma}^*_{m,l}"]{r} \arrow[two heads]{d} & \tD_{m,l,L} \arrow[two heads]{d} \\
\D_{m,l,K} \arrow["\sigma^*_{m,l}"]{r} \arrow["\phi_{m,l,K}","\sim" labl]{d} & \D_{m,l,L} \arrow["\phi_{m,l,L}","\sim" labl]{d} \\
\lPrimes_{K_{m-1}} \arrow["\theta_{m-1}^{(l')}"]{r} & \lPrimes_{L_{m-1}}
\end{tikzcd}
\end{center} 
\end{remark}

\begin{proposition}\label{EquivMap}
The map $\theta_{m-1}^{(l')}$ as defined above is compatible with $\sigma_{m-1}$, that is for all $g \in G_K^{m-1}$ we have
$\theta_{m-1}^{(l')}(g \p)=\sigma_{m-1}(g) \theta_{m-1}^{(l')}(\p)$.
\end{proposition}
\begin{proof}
Let $g \in G_K^{m-1}$ be any element, and let $\hat{g}$ be a lift of $g$ to $G_K^m$. By Proposition \ref{STMapping}, the map $\phi_{m,l,K}$ is $G_K^{m-1}$-equivariant, where the action of $G_K^{m-1}$ on $\D_{m,l,K}$ is induced by that of $G_K^m$. If we consider an equivalence class $a \in \D_{m,l,K}$, for all $F \in a$ we take $\hat{g} F \hat{g}^{-1} \in \tD_{m,l,K}$, which must be an element of the class $g a$ by definition. \\
By the definition of $\tilde{\sigma}^*_{m,l}$ it follows that $\widetilde{\sigma}^*_{m,l}(\hat{g} F \hat{g}^{-1})=\sigma_{m}(\hat{g})\tilde{\sigma}^*_{m,l}(F)\sigma_{m}(\hat{g}^{-1})$. By Proposition \ref{Diagram}, and the commutativity of the above diagram it then follows that $\sigma^*_{m,l}(g a)=\sigma_{m-1}(g)\sigma^*_{m,l}(a)$. The statement then follows immediately.
\end{proof}

\begin{remark}
With the notation of the above proposition, if we take $\p$ and $\p'$ in $\lPrimes_{K_{m-1}}$ which are conjugate by an element of a subgroup $H$ of $G_K^{m-1}$, the above result implies that their images by $\theta_{m-1}^{(l')}$ are conjugate by an element of $\sigma_{m-1}(H)$. \\
In particular, when $i \le m-1$, $H=G_K^{m-1}[i]$, and $\p$ and $\p'$ are above the same prime of $K_i$, then $\sigma(H) \subseteq G_L^{m-1}[i]$ and so the images of $\p$ and $\p'$ by $\theta^{(l')}_{m-1}$ are above the same prime of $L_i$.
\end{remark}

We reformulate this remark into the following Corollary of Proposition \ref{EquivMap}

\begin{corollary}\label{MapDown}
With the same notation as in Proposition \ref{EquivMap}, there exists a map \[ \theta_{m-2}^{(l')}: \lPrimes_{K_{m-2}} \to  \lPrimes_{L_{m-2}} \] so that the diagram
\begin{center}
\begin{tikzcd}
\lPrimes_{K_{m-1}} \arrow["\theta_{m-1}^{(l')}"]{r} \arrow{d} & \lPrimes_{L_{m-1}} \arrow{d} \\
\lPrimes_{K_{m-2}} \arrow["\theta_{m-2}^{(l')}"]{r} & \lPrimes_{L_{m-2}}
\end{tikzcd}
\end{center} 
where the vertical arrows are the natural arrows given by restriction is commutative. 
Furthermore, $\theta_{m-2}^{(l')}$ is compatible with $\sigma_{m-2}$ in the sense that 
$\theta_{m-2}^{(l')}(g \p)=\sigma_{m-2}(g) \theta_{m-2}^{(l')}(\p)$ for all $g \in G_K^{m-2}$
\end{corollary}

\begin{proposition}\label{DecMap}
Let $m \ge 1$ be an integer and $\sigma_{m+1}: G_K^{m+1} \to G_L^{m+1}$ a homomorphism of profinite groups that satisfies condition $(\dagger_l)$. Let $\theta_m^{(l')}$ be the map of primes induced by $\sigma_{m+1}$ defined as in Remark \ref{lMapPrimes}. \\
Let $\p \in \lPrimes_{K_{m}}$, and set $\q = \theta_{m}^{(l')}(\p)$. Then, if we consider the decomposition groups $D_\p \subset G_K^m$ and $D_\q \subset G_L^m$, we have $\sigma_m(D_\p) \subseteq D_\q$ and the image of $D_\p$ is non-trivial.
\end{proposition}
\begin{proof}
Let $a = \phi^{-1}_{m,l,K}(\p)$ and $h \in D_\p$. Then, by Proposition \ref{STMapping}, for all $F \in a$ we get $hFh^{-1} \in a$. It then follows by Proposition \ref{DaggerlMap} that $\sigma_{m}(h) \tilde{\sigma}^*_{m,l}(F) \sigma_{m}(h^{-1}) \in \sigma^*_{m,l}(a)$, and by commutativity of the diagram $\sigma_m(h) \in \Stab(\sigma^*_{m,l}(a))$, that is $\sigma_m(h) \in D_\q$. \\
The non-triviality follows immediately from the fact that the $l$-Sylow subgroups of $D_{\p}$ are mapped injectively by the definition of $(\dagger_l)$ and Proposition \ref{starllift}.
\end{proof}

Now that we established that $\theta_m^{(l')}$ is compatible with $\sigma_m$ and allows us to construct a map between decomposition groups, we want to check if this map is independent of the choice of $l$.

\begin{corollary}\label{DecMapSep}
Let $m \ge 1$ be an integer and $\sigma_{m+1}: G_K^{m+1} \to G_L^{m+1}$ a homomorphism of profinite groups. Assume that, for two distinct prime numbers $l_1$ and $l_2$, $\sigma_{m+1}$ satisfies both conditions $(\dagger_{l_1})$ and $(\dagger_{l_2})$. \\
Let $\p \in \Primes_{K_m}$ be a prime with residue characteristic different from $l_1$ and $l_2$. Then, $\q_1 = \theta_{m}^{(l_1')}(\p)$ and $\q_2 = \theta_{m}^{(l_2')}(\p)$ are above the same prime $\bar{\q}$ of $L_{m-1}$.
\end{corollary}
\begin{proof}
Observe that by Proposition \ref{DecMap}, $\sigma_m(D_{\p})$ is non-trivial and contained in both $D_{\q_1}$ and $D_{\q_2}$. The statement now follows from Proposition \ref{mSeparatedness}
\end{proof}

\begin{proposition}\label{MapPrimes}
Let $m \ge 0$ be an integer and let $\sigma_{m+2}: G_K^{m+2} \to G_L^{m+2}$ be a homomorphism of groups satisfying condition $(\dagger)$.\\ 
Then, we may construct a well-defined map $\theta_{m}: \Primes_{K_m} \to \Primes_{L_m}$ compatible with the induced homomorphism $\sigma_m: G_K^m \to G_L^m$ in the following sense: if $g \in G_K^{m}$, $\p \in \Primes_{K_m}$, then $\theta_m(g \p) = \sigma_{m}(g) \theta_m(\p)$.
\end{proposition}
\begin{proof}
For $l$ ranging over all prime numbers, by Proposition \ref{EquivMap} we may construct maps $\theta_{m+1}^{(l')}:\lPrimes_{K_{m+1}}\to \lPrimes_{L_{m+1}}$ compatible with $\sigma_{m+1}$, and by Corollary \ref{MapDown} we may induce a well-defined map $\theta_m^{(l')}: \lPrimes_{K_m} \to \lPrimes_{L_m}$. \\
It then follows from Corollary \ref{DecMapSep} that, for a fixed prime $\p$ of $K_m$, all of the images of $\p$ with respect to all the maps $\theta_m^{(l')}$ coincide, and we may therefore ``glue'' them together to obtain a map $\theta_{m}: \Primes_{K_m} \to \Primes_{L_m}$ , as desired. \\
The final part of the statement follows as each map $\theta_m^{(l')}$ is compatible with $\sigma_m$ by \ref{EquivMap}, therefore $\theta$ inherits this property by construction.
\end{proof}

The next few results show that this map of primes induces maps of primes which are both compatible with maps induced by $\sigma_m$ as in Remark \ref{quotients} and are not altered by these quotients.

\begin{corollary}\label{MapPrimesNF}
With the same assumptions and notations as in Proposition \ref{MapPrimes}, for all $0 \le i \le m-1$, a map $\theta_{i}: \Primes_{K_i} \to \Primes_{L_i}$ may be constructed such that, for all $0 \le i \le j \le m$, the following diagram where the vertical arrows are the natural restrictions, is commutative:
\begin{center}
\begin{tikzcd}
\Primes_{K_j} \arrow["\theta_{j}"]{r} \arrow[two heads,d] & \Primes_{L_j} \arrow[two heads, d] \\
\Primes_{K_i} \arrow["\theta_{i}"]{r} & \Primes_{L_i}
\end{tikzcd}
\end{center}
In particular, for $i=0$, we have a map $\theta: \Primes_K \to \Primes_L$. 
\end{corollary}
\begin{proof}
This follows immediately from the fact that for all $g \in G_K^m$ and integers $i$ such that $0 \le i \le m$, $\p \in \Primes_{K_i}$, $\theta_i(g \p) = \sigma_{m+1}(g) \theta_i(\p)$.\\
\end{proof}

\begin{corollary}\label{MapPrimesDown}
With the same assumptions as in Proposition \ref{MapPrimes}, let $0 \le j \le m$ be an integer. The map $\sigma_j: G_K^j \to G_L^j$ determined by $\sigma_{m+1}$ (cf. Proposition \ref{Diagram}), induces the same maps of primes $\theta_i: \Primes_{K_i} \to \Primes_{L_i}$ for all $0 \le i \le j-1$ as $\sigma_{m+1}$.
\end{corollary}
\begin{proof}
This follows immediately from Proposition \ref{MapPrimes} and Corollary \ref{MapPrimesNF} as the action of $G_K^{m+1}$ on $\Primes_{K_i}$ factors through $G_K^i$, so we get our statement from the commutativity of the diagram in Proposition \ref{Diagram}.
\end{proof}

\begin{corollary}\label{MapPrL}
With the assumptions of Proposition \ref{MapPrimes}, let $\widetilde{L}$ be the subfield of $L_m/L$ corresponding to $\sigma_m(G_K^m)$. Then, there exists a unique map $\tilde{\theta}: \Primes_K \to \Primes_{\widetilde{L}}$ such that $\theta: \Primes_K \to \Primes_L$ factors through $\Primes_{\widetilde{L}}$. 
\end{corollary}
\begin{proof}
Let $\hat{\p}_1, \hat{\p}_2$ be any two primes of $K_m$ above the same prime $\p \in \Primes_K$. Then, there exists $g \in G_K^m$ such that $g \hat{\p}_1 = \hat{\p}_2$, and it follows by Proposition \ref{MapPrimes} that the images of $\hat{\p}_1$ and $\hat{\p}_2$ are above the same prime in the field fixed by  $\sigma_m(G_K^m)$, that is $\widetilde{L}$. The statement now follows immediately.
\end{proof}

\begin{remark}
Observe that, if we take a prime $\p \in \Primes_{K_{m+1}}$ with decomposition group $D_{\p} \subset G_K^{m+1}$ we show in Proposition \ref{DecMap} that $\sigma_{m+1}(D_{\p})$ is mapped to the decomposition group $D_{\q} \subset G_L^{m+1}$ of some prime $\q \in \Primes_{L_{m-1}}$ however $\q$ is not necessarily unique (a priori). \\
If we however were to consider the primes $\bar{\p}$ of $K$ below $\p$ and $\bar{\q}$ of $L$ below $\q$, we may take the decomposition groups determined up to conjugation and in this case the uniqueness of $\q$ would hold. \\ Depending on the information we are looking for it then might be useful to just consider decomposition groups up to conjugation (as we do later in Proposition \ref{NormMap}). \end{remark}

We end this section with a statement that shows that these maps of primes preserve the residue characteristic of the primes.

\begin{proposition}\label{MapResChar}
Let $m \ge 1$ be an integer, and $\sigma_{m+2}: G_K^{m+2} \to G_L^{m+2}$ a homomorphism satisfying condition $(\dagger)$.
Then, for all $0 \le i \le m$, let us consider the maps $\Primes_{K_i} \surj \Primes_{\Q}$ and $\Primes_{L_i} \surj \Primes_{\Q}$ given by residue characteristics. 
Then, we have a commutative diagram
\begin{center}
\begin{tikzcd}
\Primes_{K_i} \arrow["\theta_{i}",rr] \arrow[two heads,dr] && \Primes_{L_i} \arrow[two heads, dl] \\
&\Primes_{\Q}&
\end{tikzcd}
\end{center}
\end{proposition}
\begin{proof}
Let $p \in \Primes_{K_i}$ and $\q=\theta_i(\p)$. Then, for some extension $\hat{\p}$ of $\p$ to $K_{m+2}$ and $\hat{\q}$ of $\q$ to $L_{m+2}$, we have that a subgroup satisfying condition $(\star_l)$ of $D_{\hat{\p}}$ is mapped to a subgroup satisfying condition $(\star_l)$ of $D_{\hat{\q}}$ by $\sigma_{m+2}$. Then, $\ch(\hat{\p})$ and $\ch(\hat{\q})$ are both different from $l$. Repeating this argument for all $l \neq \ch(\p)$ shows that indeed $\ch(\p)=\ch(\q)$.
\end{proof}

\begin{definition}\label{daggerplus}
For $m \ge 2$, let $\sigma_m: G_K^m \to G_L^m$ be a homomorphism of profinite groups satisfying condition $(\dagger)$ and such that there exists a map of primes $\theta: \Primes_K \to \Primes_L$ such that for all $\p \in \Primes_K$, $\ch(\theta(\p))=\ch(\p)$ and $\sigma_m(D_{\p}) \subseteq D_{\q}$. \\
Then, we will say that $\sigma_m$ satisfies condition $(\dagger^+)$, and that $\theta$ is the map of primes associated to $\sigma_m$.
\end{definition}

\begin{remark}\label{factormap}
Let $\sigma_m: G_K^m \to G_L^m$ be a homomorphism of profinite groups satisfying property $(\dagger^+)$, $\theta: \Primes_K \to \Primes_L$ be the map of primes associated to it and let $\widetilde{L}$ be the subfield of $L_m/L$ corresponding to $\sigma_m(G_K^m)$. \\
For $\p \in \Primes_K$, all the possible choices of $D_{\p}$ up to conjugation are mapped to conjugate subgroups in $G(L_m/\widetilde{L})=\sigma_m(G_K^m)$. Then, it follows that for some prime $\tilde{\q} \in \Primes_{\widetilde{L}}$, we may define a map $\tilde{\theta}(\p)=\tilde{\q}$. It then follows that we can construct $\tilde{\theta}: \Primes_K \to \Primes_{\widetilde{L}}$ as in Corollary \ref{MapPrL}, and $\tilde{\theta}$ is such that $\theta$ factors through $\tilde{\theta}$ as in the diagram \begin{center}
\begin{tikzcd}
& \Primes_{\widetilde{L}} \arrow[d, two heads] \\
\Primes_K \arrow[r,"\theta"] \arrow[ur,"\tilde{\theta}"] & \Primes_L 
\end{tikzcd}
\end{center}
where the vertical arrow is the natural restriction in the extension $\widetilde{L}/L$.
We will then say this map $\tilde{\theta}$ is the factor map of primes associated to $\sigma_m$.
\end{remark}

\begin{remark}\label{EndowedMap}
Let $m \ge 2$ be an integer, and let $\sigma_{m+1}: G_K^{m+1} \to G_L^{m+1}$ be a homomorphism of profinite groups satisfying condition $(\dagger)$. Then, the induced homomorphism $\sigma_m: G_K^m \to G_L^m$ satisfies condition $(\dagger^+)$, and the map $\theta: \Primes_K \to \Primes_L$ associated to it coincides with the map $\theta$ obtained in Proposition \ref{MapPrimesNF} starting with these assumptions.
\end{remark}

\section{Existence in the m-step Hom-Form}

Throughout this section, we take an integer $m \ge 2$ and assume we have a homomorphism of profinite groups $\sigma_{m+2}: G_K^{m+2} \to G_L^{m+2}$ that satisfies condition $(\dagger^+)$ (cf. Definition \ref{daggerplus}). Then, by definition, we have a map $\theta: \Primes_K \to \Primes_L$ that preserves the residue characteristic of primes. We want to study whether this map also induces relations on other invariants associated to primes, namely we'll look at the inertia degree and the norm. \\ We are going to need the following result, which we will apply to decomposition groups. 

\begin{lemma}\label{localroots}
Let $p$ be a prime number, and $k,k'$ be $p$-adic local fields. Let us assume we have a homomorphism of profinite groups $\alpha: G_k \to G_{k'}$ and that there exists a prime number $l \neq p$ such that $\alpha$ restricts to an injection on $l$-Sylow subgroups of $G_k$. \\ Then, $|\mu(k)(l)| \ge |\mu(k')(l)|$. In particular, $\mu_l \subset k'$ implies $\mu_l \subset k$.
\end{lemma}
\begin{proof}
First, assume that $\mu_l \not\subset k'$. Then, the statement is trivially true. We may then assume without loss of generality that $k'$ does indeed contain the $l$-th roots of unity. In this case, the maximal pro-$l$-quotient $G_{k'}(l)$ of $G_{k'}$ is an extension of $\Z_l$ by $\Z_l$ (cf. Proposition 7.5.9, \cite{NWS}). \\
Let $G_{k,l}$ be an $l$-Sylow subgroup of $G_k$, $H=\alpha(G_k)$, and let $H_l$ be the $l$-Sylow subgroup of $H$ such that $\alpha(G_{k,l}) = H_l$. By our assumption, $cd\,H_l = 2$. In particular, $H_l$ is open in an $l$-Sylow subgroup $G_{k',l}$ of $G_{k'}$. If we consider the quotient $G_{k'} \surj G_{k'}(l)$, we observe that $G_{k',l}$ must map surjectively to $G_{k'}(l)$, and isomorphically by Proposition \ref{zlinj}. Then, the image of $H_l$ is an open subgroup $U$ of $G_{k'}(l)$. We may then take the inverse image of $U$ by the quotient $G_{k'} \to G_{k'}(l)$ to get an open subgroup $G_F$ of $G_{k'}$ such that $H_l$ is an $l$-Sylow subgroup of $G_F$. Let $F$ be the finite extension of $k'$ corresponding to this subgroup. By construction, we also have $H \subseteq G_F$. \\ We may then take the composite $G_{k} \to G_{F} \surj G_F(l)$, and observe that since the maps $G_{k,l} \surj G_{F,l} \surj G_F(l)$ are surjective, this composite is also surjective. Then, it must factor through $G_k(l)$, and we have a commutative diagram
\begin{center}
\begin{tikzcd}
G_{k,l} \arrow[r,two heads] \arrow[d,two heads] & G_{F,l} \arrow[d,two heads] \\
G_k(l) \arrow[r,two heads] & G_F(l) 
\end{tikzcd}
\end{center}
which shows that $G_k(l)$ is of cohomological dimension $2$, and so $\mu_l \subset k$, and all the arrows in the diagram are isomorphisms. We may then also naturally induce an isomorphism $G_k(l)^{\ab} \isomto G_F(l)^{\ab}$. We may then conclude by (\cite{NWS}, Proposition 7.5.9), that $\mu(k)(l)$ is isomorphic $\mu(F)(l)$ and the statement then follows immediately, as $F$ is an extension of $k'$. 
\end{proof}

\begin{corollary}\label{localroots2}
Let $m \ge 2$ be an integer. In the previous statement, replace $G_k$ and $G_{k'}$ with their maximal $m$-step solvable quotients $G_k^m$ and $G_{k'}^m$. Then, the statement still holds.
\end{corollary}
\begin{proof}
Since by our assumptions $m \ge 2$, we have that $L = \ker(G_{k} \to G_{k}^m)$ is a pro-$p$-group  (cf. \cite{S-T}, Proposition $1.1.$vii). Furthermore, if we consider the maximal pro-$l$-quotient $G_k^m(l)$ of $G_k^m$, the quotient map $G_k \surj G_k^{m}(l)$ must factor through $G_k(l)$ by maximality, and the map $G_k(l) \to G_k^m(l)$ has naturally pro-$l$ kernel. Then, for appropriate normal subgroups $V$ and $W$ of (respectively) $G_k$ and $G_k^m$, we have two exact sequences which fit in the following diagram:
\begin{center}\begin{tikzcd}
1 \arrow[r] & V \arrow[r] \arrow[d, two heads] & G_k \arrow[r, two heads] \arrow[d, two heads] & G_k(l) \arrow[d] \arrow[r] & 1 \\
1 \arrow[r] & W \arrow[r] & G_k^m \arrow[r, two heads] & G_k^m(l) \arrow[r] & 1
\end{tikzcd}\end{center}
However, $L$ is pro-$p$, and it must map surjectively to the kernel of $G_k(l) \to G_k^m(l)$, which is pro-$l$. The latter is then trivial. A proof analogous to the one of Lemma \ref{localroots} allows us to conclude that we have an isomorphism $G_k^m(l) \isomto G_F^m(l)$ for some finite subextension $F$ of $k'_m/k'$, and the statement follows as desired.
\end{proof}

\begin{proposition}\label{NormMap}
Let $m \ge 2$ be an integer and $\sigma_{m}: G_K^{m} \to G_L^{m}$ be a homomorphism of profinite groups satisfying condition $(\dagger^+)$ (see Definition \ref{daggerplus}). Let $\widetilde{L}$ be the subextension of $L_m/L$ corresponding to $\sigma_m(G_K^m)$. \\ Consider the map of primes $\theta; \Primes_K \to \Primes_L$ and the factor map of primes $\tilde{\theta}: \Primes_K \to \Primes_{\widetilde{L}}$ (see Remark \ref{factormap}) associated to $\sigma_m$. \\
Let $\p \in \Primes_K$ and $\q = \theta(\p)$. Then, $N \p \ge N \q$. \\
More generally, let $\tilde{\q} = \tilde{\theta}(\p)$, and $L'$ a finite subextension of $\widetilde{L}/L$. Then, for the prime $\q' = \tilde{\q} \cap L'$, we have $N(\p) \ge N(\q')$
\end{proposition}
\begin{proof}
Let us denote $p = \ch(\p)$ ($= \ch(\q)$ by Definition \ref{daggerplus}). We may take $D_{\p} \subset G_K^m$ a chosen decomposition group (unique up to conjugation) at $\p$ in $K_m/K$, and let $D_{\q} \subset G_L^m$ be a decomposition group at $\q$ such that $\sigma_m(D_{\p}) \subseteq D_{\q}$ (cf. Definition \ref{daggerplus} again). For any $l \neq p$, by our assumptions, the restriction of $\sigma_m$ to any $l$-Sylow subgroup of $D_{\p}$ gives an injective map. Then, we may apply Corollary \ref{localroots2}, and obtain that $\mu(K_{\p})(l) \supseteq \mu(L_{\q})(l)$. \\
Recall that $N(\p) = p^{f_{\p}}$ is equal to $1$ plus the cardinality of the prime-to-$p$ torsion in $D_{\p}^{\ab}$, which corresponds to the group of prime-to-$p$-power roots of unity in contained in the localisation $K_{\p}$. An iteration of the above argument for all $l \neq p$ shows that $\mu(L_{\q})^{(p')}$ is then a subgroup of $\mu(K_{\p})^{(p')}$, which also implies $N(\p) \ge N(\q)$ as desired. \\
We now need to prove the second part of the statement. Let $L'$ be any fixed finite subextension of $\widetilde{L}/L$, and let us take an extension $\q'$ of $\q$ to $L'$ determined by $D_{\q} \cap G(L_m/L')$. We will denote this intersection by $D_{\q'}$. \\
We naturally have a quotient $\mathcal{G}_{\q'} := G_{L'_{\q'}} \surj D_{\q'}$. Our goal is to see that, analogously to $G_{k'}^m$ in the proof of Corollary \ref{localroots}, $D_{\q'}(l) \isom \mathcal{G}_{q'}(l)$. \\
Let us first consider the exact sequence \[ 1 \to V \to \mathcal{G}_{\q'} \to D_{\q'} \to  1 \] where $V$ is the appropriate normal subgroup of $\mathcal{G}_{\q'}$.
If we choose $l \neq p$ a prime number, we naturally have an $l$-Sylow subgroup of $\G_{\q'}$ must map surjectively onto an $l$-Sylow subgroup of $D_{\q'}$. However, we know that by our assumptions $\cd_l D_{\q'} = 2$ as an $l$-Sylow subgroup of $D_{\q'}$ must be a subgroup of an $l$-Sylow subgroup of $D_{\q}$ (which is of $\cd_l$ $2$) and contains the image by $\sigma$ of an $l$-Sylow subgroup $D_{\p,l}$ of $D_{\p}$ (which is also of $\cd_l$ $2$). It then follows from Lemma \ref{zlinj} that an $l$-Sylow of $\G_{\q'}$ is mapped isomorphically by the quotient map $\mathcal{G}_{\q'} \surj D_{\q'}$ to an $l$-Sylow of $D_{\q'}$. It then also follows that an $l$-Sylow subgroup $V_l$ of $V$ is trivial, and repeating this argument for all $l \neq p$ shows that $V$ is a pro-$p$-group. \\
Observe that $\ker(\G_{\q'} \to \G_{\q'}^{\tr})$ is the maximal pro-$p$ normal subgroup of $\G_{\q'}$ (\cite{NWS}, Corollary 7.5.7), which means that we have a natural surjective map $D_{\q'} \surj \G_{\q'}^{\tr}$. Since the maximal pro-$l$ extension $L'_{\q'}(l)/L'_{\q'}$ is tamely ramified, it follows that $\G_{\q'}(l)$ is also a quotient of $D_{\q'}$, and maximality then shows $D_{\q'}(l) \isom \G_{\q'}(l)$. \\	
We may now proceed with an argument analogous to the one in the proof of Lemma \ref{localroots} to show that $\mu(K_{\p})(l) \supseteq \mu(L'_{\q'})(l)$. We then conclude $N(\p) \ge N(\q')$, and since $L'$ was chosen arbitrarily as a subextension of $\widetilde{L}/L$, the second part of the statement follows. \end{proof}

\begin{proposition}\label{cofiniteprimes}
With the same assumptions and notations as in Proposition \ref{NormMap}, only finitely many primes of $L$ are not in the image of $\theta$, and $[K:\Q] \ge [L:\Q]$.
\end{proposition}
\begin{proof}
Assume by contradiction there are infinitely many primes of $L$ that are not in the image of $\theta$. Since $L$ has a finite number of ideal classes, there must be an ideal class of $L$ containing infinitely many of these primes, so let us denote them by $\q_0,\q_1,\q_2,...$. Furthermore, for all $i \ge 1$ the ideal $\q_0/\q_i$ is principal, and so is generated by some element $\alpha_i$. \\
Then, we may consider the infinite abelian extension $L'=L(\sqrt{\alpha_1},\sqrt{\alpha_2},...)$ of $L$, where only the primes above $2$ and the $\q_i$ may ramify, and the abelian extension $K'$ of $K$ corresponding to it by $\sigma_m$. Observe that if a finite prime of $K$ ramifies in $K'/K$, its image by $\theta$ must also ramify in $L'/L$. As the $\q_i$ are by definition not in the image of $\theta$, by Proposition \ref{MapResChar} the only primes of $K$ that may ramify in $K'/K$ are primes with residue characteristic $2$. As we have an injective map $G(K'/K) \inj G(L'/L)$ by Remark \ref{quotients}, and $L'/L$ is obtained as a composite of (infinitely many) $\Z/2\Z$-extensions of $L$, $K'$ is obtained as a composite of $\Z/2\Z$-extensions of $K$ too. However, as the only finite primes that may ramify in $K'/K$ are the primes above $2$, and there are only finitely many quadratic extensions of $K$ that satisfy this property (cf. \cite{NeukirchANT}, Theorem 2.16), it follows that $K'/K$ needs to be a finite extension. \\
As $G(K'/K) \isom G(\widetilde{L} L'/\widetilde{L})$, it follows that $\widetilde{L} L'/\widetilde{L}$ is a finite extension as well. We may then consider the field $L'' = L' \cap \widetilde{L}$, and $G(L'/L'') \isom G( \widetilde{L} L'/ \widetilde{L})$, so $L''$ corresponds to a finite subgroup of $G(L'/L)$. In particular the abelian extension $L'' / L$ is infinite. \\
Let $\p \in \Primes_K$ be a prime such that $\ch(\p)$ is odd and $f_\p = 1$. Then $\q=\theta(\p)$ is unramified in $L'/L$, and by \ref{NormMap} $\ch(\p) = N \p = N \q$. This also holds if we replace $\q$ with any extension $\q'$ of $\q$ to any finite subextension of $L''/L$, since $L'' \subset L'$. This then implies that $\q$ splits completely in any finite subextension of $L''/L$ and so it also does in $L''/L$. Observe also that only up to $[K:\Q]$ primes of $K$ may map to a same prime $\q$ of $L$ as a consequence of Proposition \ref{MapResChar}. \\
Let us now denote by $A$ the set of all non-archimedean primes $\p$ of $K$ with odd residue characteristic and $f_\p=1$. As $A$ contains all the non-archimedean primes of $K$ that split completely in $K/\Q$, up to the finite set of primes with residue characteristic $2$, it has positive Dirichlet density, in particular
\[ \delta_K(A) = \lim_{s \to 1^+} \frac{\sum_{\p \in A} N \p^{-s}}{\log (1/(s-1))} \ge 0. \]
The image $\theta(A)$ of $A$ in $L$ then is a set of density \[ \delta_L(\theta(A)) = \lim_{s \to 1^+} \frac{\sum_{\q \in \theta(A)} N \q^{-s}}{\log (1/(s-1))} \ge \frac{\delta_K(A)}{[K:\Q]} \] as $N \p = N \q$ for $\p \in A$ and $\q = \theta(\p)$ and by the above observation $\theta^{-1}(\q)$ contains at most $[K:\Q]$ elements. Then, for any finite subextension $L^{\circ}$ of $L''/L$, the primes of $\theta(A)$ will split completely in $L^{\circ}/L$. However, since $L''/L$ is infinite, we may take $L^{\circ}/L$ of arbitrarily large degree so that $\delta_L(\theta(A)) > [L^{\circ}:L]^{-1}$. We then have it is impossible for a set of primes of density $\delta_L(\theta(A))$ to split in $L^{\circ}/L$ and we then get a contradiction, so all primes of $L$ except finitely many are in the image of $\theta$. \\
We now prove the second assertion in the statement. Let $p$ be a prime number such that any prime of $L$ above $p$ is unramified in $L/\Q$, and all the primes of $L$ that extend $p$ are in the image of $\theta$. Observe that we may take such a prime $p$ as there are infinitely many prime numbers that are unramified in $L/\Q$, and by the first part of the proposition only a finite number of them may have an extension to $L$ that is not in the image of $\theta$. For a fixed prime number $p$ with these properties, we denote by $\q_i \in \Primes_{L}$ all the possible extension of $p$ to $L$, indexed by $i \in \{1,..,k\}$ for an appropriate finite integer $k$. By our assumption, for each $i \in \{1,..,k\}$ there exists $\p_i \in \Primes_K$ such that $\theta(\p_i) = \q_i$. The inequality $f_{\p} \ge f_{\q}$ given by Proposition \ref{NormMap} then gives inequalities \[ [K:\Q] \ge \sum_{i=1}^k e_{\p_i} f_{\p_i} \ge \sum_{i=1}^k f_{\q_i} = [L:\Q] \] as $e_{\q_i}=1$ for all $i$ by our assumption that $p$ is unramified in $L/\Q$. We may then conclude $[K:\Q] \ge [L:\Q]$ as desired.
\end{proof}

\begin{proposition}\label{isopen} 
Let $m \ge 2$ and $\sigma_{m+2}: G_K^{m+2} \to G_L^{m+2}$ be a homomorphism of groups satisfying condition $(\dagger)$, and assume that the subextension $\widetilde{L}$ of $L_{m+2}/L$ corresponding to $\sigma_{m+2}(G_K^{m+2})$ is a subfield of $L_{m-1}$. Then, $\sigma_m$ has open image.
\end{proposition}
\begin{proof}
First, observe that by Remark \ref{EndowedMap}, $\sigma_{m+1}$ satisfies condition $(\dagger^+)$, and let $\tilde{\theta}: \Primes_K \to \Primes_{\widetilde{L}}$ be the map associated to it. \\
Let $L'$ be a finite subextension of $\widetilde{L}/L$. Then, $L'_3 \subseteq L_{m+2}$, and we may see $G(L_{m+2}/\widetilde{L})$ as a subgroup of $G(L_{m+2}/L')$. By construction, we have the quotient $G(L_{m+2}/L') \surj G(L'_3/L)$, and by Lemma \ref{DiagramLemma} it follows that starting from $\sigma_{m+2}$ we may induce a homomorphism of profinite groups $\sigma'_{3}: G_K^3 \to G_{L'}^3$ which satisfies condition $(\dagger)$, and $\sigma'_2: G_K^2 \to G_L^2$ must satisfy condition $(\dagger^{+})$ by Remark \ref{EndowedMap} as $\sigma_{3}$ factors through $\sigma'_3$.\\
Then, $\sigma'_3$ induces a map $\theta': \Primes_K \to \Primes_{L'}$ as in Proposition \ref{MapPrimesNF}, which is the map of primes associated to $\sigma'_2$ and this map by construction must coincide with the composition of $\tilde{\theta}$ with the natural map $\Primes_{\widetilde{L}} \to \Primes_{L'}$. \\
This map $\theta'$ satisfies the inequality of Proposition \ref{NormMap}. We may then apply Proposition \ref{cofiniteprimes} to show that $[K:\Q] \ge [L':\Q]$. As we can do this for arbitrary finite subextensions $L'$ of $\widetilde{L}/L$, it must follow that $[K:\Q] \ge [\widetilde{L}:\Q]$, and so $\sigma_m$ is open.
\end{proof}

\begin{remark}\label{RemarkOpen}
The existence of an embedding $\tau: L_m \to K_m$ such that $\tau \sigma_m(g) = g \tau$ implies that the homomorphism $\sigma_m$ has open image, as the field $\widetilde{L}$ corresponding to $\sigma_m(G_K^m)$ must embed in $K$. It is however interesting to note that the two conditions in Proposition \ref{isopen} together give us that $\sigma_m$ has open image before knowing about the existence of $\tau$.
\end{remark}

\begin{remark}
The condition $\widetilde{L} \subseteq L_{m-1}$ is group theoretical, as it is equivalent to $\sigma_m(G_K^m) \supseteq G_L^m[m-1]$.
\end{remark}

The following two results are transposition of Lemma 5 and Lemma 6 in \cite{UchidaHom} to the $m$-step case.

\begin{proposition}\label{Existence1}
Let $m \ge 2$, and $\sigma_m: G_K^m \to G_L^m$ be a homomorphism of profinite groups satisfying condition $(\dagger^+)$. Assume $K$ and $L$ are contained in the same separable closure $\overline{\Q}$ of $\Q$. \\
Let $M/\Q$ be a Galois extension containing both $K$ and $L$. Let $H = G(M/\Q)$, and consider its subgroups $H_1 = G(M/K)$ and $H_2=G(M/L)$. Then, every element of $H_1$ is conjugate to an element of $H_2$ in $H$. \\
This result also holds when replacing $K$ and $L$ with finite subextensions $K'$ of $K_{m+1}/K$ and $L'$ of $L_{m+1}/L$ respectively such that $L'$ corresponds to $K'$ by $\sigma_{m+1}$.
\end{proposition}
\begin{proof}
Take $h \in H_1$, and let $\Pc \in \Primes_M$ be a prime unramified in $M/\Q$ on which $h$ induces the Frobenius automorphism. Let $\p$ be the prime of $K$ below $\Pc$, and set $\theta(\p)=\q$. As $\Pc$ is unramified in $M/\Q$, $\p$ is unramified in $K/\Q$. By Proposition \ref{NormMap} it follows that $f_\q = 1$. We may then take $\mathfrak{Q} \in \Primes_M$ above $\q$ and, as $\ch(\p)=\ch(\q)$, there must be an element $t \in H$ such that $t \Pc = \mathfrak{Q}$. Then, $t h t^{-1} \in H_2$. \\
We now prove the second part of the statement. Assume that $M$ contains both $K'$ and $L'$. Naturally, $G(M/K') \subseteq G(M/K)$, so $h' \in G(M/K') \implies h' \in H_1$. Using the same notation as  above, let $\p' = \Pc \cap K'$, and for some prime $\hat{\p}$ above $\p'$ in $K_{m+1}$, and a prime number $l \neq \ch(\p)$, consider the prime $\hat{\q} = \theta_{m+1}^{(l')}(\hat{\p})$, and $\q' = \hat{\q} \cap L'$. Then, by the same argument as above, we may take a prime $\mathfrak{Q}' \in \Primes_M$ above $\q'$, and an element $t' \in H$ such that $t' \Pc = \mathfrak{Q}'$, so that for $h' \in G(M/K')$ we have $t'h't'^{-1} \in G(M/L')$. 
 \end{proof}

\begin{proposition}\label{Existence2}
Let $m \ge 2$, and $\sigma_m: G_K^m \to G_L^m$ be a homomorphism of profinite groups satisfying condition $(\dagger^+)$. Assume $K$ and $L$ are contained in the same separable closure $\overline{\Q}$ of $\Q$. \\ Let $L'$ be a finite Galois extension of $L$ contained in $L_m$ and let $K'$ be the finite Galois extension of $K$ contained in $K_m$ corresponding to $L'$ by $\sigma_m$. Let $M \subset \overline{\Q}$ be a finite Galois extension of $\Q$ containing $L'$ and $K$. Then $M$ also contains $K'$.
\end{proposition}
\begin{proof}
Let $p$ be a prime number that splits completely in $M/\Q$. Then, all the primes above $p$ in $L$ split completely in $L'/L$. Since almost all primes of $L$ are in the image of $\theta$, we may assume that all the primes of $L$ above $p$ are in the image of $\theta$. \\
The inverse image by $\theta$ of every prime of $L$ above $p$ is a prime of $K$ which must also be unramified in $K/\Q$ as $K \subset M$ and by assumption $p$ splits completely in $M/\Q$, so all the primes of $K$ above $p$ are unramified in $K/\Q$. Furthermore, $L'$ corresponds to $K'$, and all the primes above $p$ split completely in $L'/L$, so they also split completely in $K'/K$. \\
We may then conclude that, up to the finite set of primes of $K$ above the primes $p$ for which there is a prime of $L$ above $p$ not in the image of $\theta$, every prime that splits completely in $M/K$ splits completely in $K'/K$, and we may conclude that $K' \subseteq M$.
\end{proof}

We are now able to prove the following theorem, which proves a conditional existence for the $m$-Step Hom-Form. The proof follows roughly the same methodology used by Uchida in \cite{UchidaIsom} and \cite{UchidaHom}, and by Sa\"idi and Tamagawa in \cite{S-T}. 

\begin{theorem}\label{Existence}
Let $m \ge 1$ be a positive integer, and let $\sigma_{m+1}: G_K^{m+1} \to G_L^{m+1}$ be a homomorphism of profinite groups which satisfies condition $(\dagger^+)$. Consider the induced homomorphism of profinite groups $\sigma_m: G_K^m \to G_L^m$.
Then, there exists an embedding of fields $\tau: L_m \to K_m$ such that \[ \tau \sigma_m(g) = g \tau \] for all $g \in G_K^m$.
\end{theorem}
\begin{proof}
Let $L'$ be a finite Galois subextension of $L_m/L$, and let $K'$ be the finite Galois subextension of $K_m/K$ corresponding to it by $\sigma_m$. Set $H_1 = G(K'/K)$ and $H_2 = G(L'/L)$. Recall that the map $\sigma: H_1 \inj H_2$ induced from $\sigma_m$ by quotients is injective by Remark \ref{quotients}. \\
Let $N$ be a finite Galois extension of $\Q$ for which we have embeddings $K' \inj N$ and $L' \inj N$, and let us consider the images of these embeddings canonically identified with $K'$ and $L'$. Let us also denote $\Omega$ a separable closure of $\Q$ containing $N$. \\ We then have the Galois groups $H=G(N/\Q)$, $S_1=G(N/K)$, $S_2 = G(N/L)$, $T_1 = G(N/K')$ and $T_2 = G(N/L')$. Consider a set of generators $h_{1,1},...,h_{1,n}$ for $H_1$, and set $h_{2,i}=\sigma(h_{1,i})$ for $i=1,...,n$. As $H_1=S_1/T_1$, for all $i=1,...,n$ we may define an element $s_{1,i} \in S_1$ such that $s_{1,i} T_1 = h_{1,i}$, and similarly $s_{2,i} \in S_2$ such that $s_{2,i} T_2 = h_{2,i}$. For each $i=1,...,n$ we then define a subgroup $S_{1,i}$ generated by $s_{1,i}$ and $T_1$, and similarly, $S_{2,i}$ is the subgroup generated by $s_{2,i}$ and $T_2$. Let us also set $S_{1,0}=T_1$ and $S_{2,0}=T_2$. \\ 
Consider the subfield $N_{1,i}$ of $N$ corresponding to $S_{1,i}$ and the subfield $N_{2,i}$ of $N$ corresponding to $S_{2,i}$. By the above construction, $S_{1,i}/T_1 \isom G(K'/N_{1,i})$ is a cyclic subgroup of $H_1$ generated by $h_{1,i}$, and $S_{2,i}/T_2$ is a cyclic subgroup of $H_2$ generated by $h_{2,i}=\sigma(h_{1,i})$. We then have that $\sigma$ restricts to a surjective homomorphism $S_{1,i}/T_1 \surj S_{2,i}/T_2$. However, as $\sigma$ is also injective, this is an isomorphism. Furthermore, $N_{2,i}$ must correspond to $N_{1,i}$ by $\sigma_m$. \\
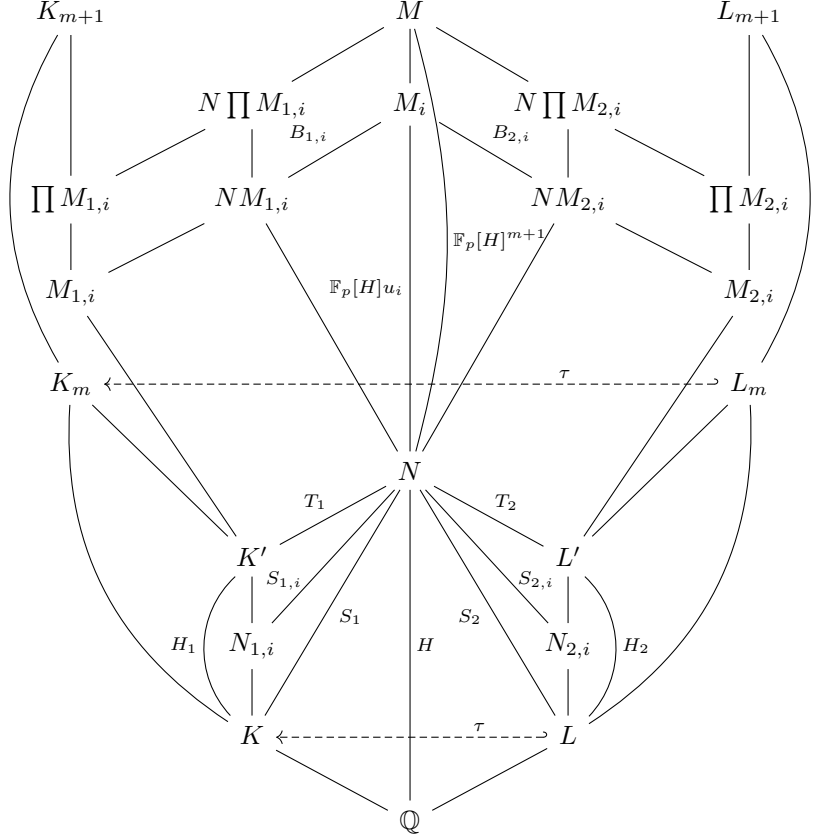
\begin{center}
\begin{tikzcd}
K_{m+1}\arrow[dash,bend right=30]{dddd}\arrow[dash]{dd}&&M\arrow[dash]{d}\arrow[dash]{dl}\arrow[dash]{dr}\arrow[dash,bend left=15]{ddddd}{\F_p[H]^{m+1}}&&L_{m+1} \arrow[dash]{dd}\arrow[dash,bend left=30]{dddd}\\
&N\prod M_{1,i}\arrow[dash]{dl}\arrow[dash]{d}&M_i\arrow[dash,swap]{dl}{B_{1,i}}\arrow[dash]{dr}{B_{2,i}}\arrow[dash,swap]{dddd}{\F_p[H]u_i}&N\prod M_{2,i}\arrow[dash]{d}\arrow[dash]{dr}&\\
\prod M_{1,i}\arrow[dash]{d}&N M_{1,i} \arrow[dash]{dl} \arrow[dash]{dddr}&& N M_{2,i} \arrow[dash]{dr}\arrow[dash]{dddl}&\prod M_{2,i}\arrow[dash]{d} \\
M_{1,i} \arrow[dash]{dddr}&&&& M_{2,i} \arrow[dash]{dddl} \\
K_m \arrow[dash, bend right=30]{ddddr} \arrow[dash]{ddr} &&&&L_m \arrow[dash, bend left=30]{ddddl} \arrow[swap,hook',dashed,near start,"\tau"]{llll} \arrow[dash]{ddl} \\&&
N \arrow[dash,swap,"T_1"]{dl} \arrow[dash,swap,close,near end,"S_{1,i}"]{ddl}  \arrow[dash,"S_{1}"]{dddl}\arrow[dash,"T_2"]{dr}\arrow[dash,close,near end,"S_{2,i}"]{ddr}\arrow[dash,swap,"S_2"]{dddr}\arrow[dash]{dddd}{H} &&\\
&K'\arrow[dash]{d}\arrow[dash,bend right=45,swap]{dd}{H_1}&&L'\arrow[dash]{d}\arrow[dash,bend left=45]{dd}{H_2}&\\
&N_{1,i}\arrow[dash]{d}&&N_{2,i}\arrow[dash]{d}& \\
&K\arrow[dash]{dr}&&L\arrow[dash]{dl} \arrow[swap,hook',dashed,near start,"\tau"]{ll}& \\
&&\Q&&
\end{tikzcd}
\end{center}
We may now take a prime number $p$ such that $p \equiv 1$ (mod $|H|$) and $p > |H|^2$, and consider a split group extension \[1 \to \F_p[H]^{n+1} \to E \to H \to 1.\]
Then, by (\cite{NWS}, Proposition 9.2.9) there exists a Galois extension $M$ of $\Q$ containing $N$ such that $G(M/N) = \F_p[H]^{n+1}$ and $G(M/\Q)=E$. We may also take a set of elements $u_0,...,u_n$ of $\F_p[H]^{n+1}$ so that we may write \[\F_p[H]^{n+1} = \bigoplus_{i=0}^n \F_p[H]u_i\] and for every $i=0,...,n$, we consider the subextension $M_i$ of $M/N$ determined by the subgroup $\oplus_{j \neq i} \F_p[H]u_j$ of $\F_p[H]^{n+1}$. By construction, it follows that $M_i$ is a Galois extension of $\Q$, and $G(M_i/\Q)$ is determined by the split extension \[1 \to \F_p[H]u_i \to G(M_i/\Q) \to H \to 1.\]
The following construction is given for all $i=0,...,n$.
Let $\chi_i$ be a character of $S_{1,i}/T_i$ of order $|S_{1,i}/T_i|$ (observe that when $i=0$, $S_{1,0}/T_1$ is trivial and so is $\chi_0$), which we may consider as valued in $\F_p$. Since $\sigma$ induces an isomorphism $S_{1,i}/T_1 \isom S_{2,i}/T_2$, we may induce by $\chi_i \sigma^{-1}$ a character of $S_{2,i}/T_2$, which we will denote $\chi_i'$. \\
Let $M_{2,i}/L'$ be the maximal $p$-extension of $L'$ contained in $M_i$ where the operation given by the action of $S_{2,i}/T_2$ on $G(M_{2,i}/L')$ coincides with the scalar multiplication of the values of $\chi_i'$. \\ As this is an abelian extension of $L' \subset L_m$, we get $M_{2,i}$ can be canonically identified with a subfield of $L_{m+1}$. We may then also consider the subextension of $K_{m+1}/K$ corresponding to $M_{2,i}$ by $\sigma_{m+1}$, which we will denote by $M_{1,i}$, and as $M_{2,i}$ is an extension of $L'$, it follows that $M_{1,i}$ is an extension of $K'$. By Proposition \ref{Existence2} as $M_i$ contains $M_{2,i}$ and $K'$, it also contains $M_{1,i}$. \\
Furthermore, we have injective maps $\sigma'_i: G(M_{1,i}/N_{1,i}) \inj G(M_{2,i}/N_{2,i})$ induced by $\sigma_{m+1}$ by quotients (see Remark \ref{quotients}). However, since $N_{2,i}$ corresponds to $N_{1,i}$ by $\sigma_m$, this is in fact an isomorphism. It follows that the operation given by the action of $S_{1,i}/T_1$ on $G(M_{1,i}/N_{1,i})$ must coincide with scalar multiplication by the values of $\chi_i$. \\
We may then consider the composite field $N M_{1,i}$, which is a subextension of $M_i/N$ and so is corresponding to a subgroup $B_{1,i}$ of $F_p[H]u_i = G(M_i/N)$. Similarly, we define the subgroup $B_{2,i}$ corresponding to $N M_{2,i}$. \\
By construction, $G(M_{1,i}/K')$ and $G(NM_{1,i}/N)=\F_p[H]u_i / B_{1,i}$ are isomorphic as $S_{1,i}/T_1$-modules, thus if we take an element $b_{1,i} \in B_{1,i}$, we may construct a subgroup $(b_{1,i} - \chi_i(b_{1,i}))\F_p[H]u_i$ of $B_{1,i}$. We may then consider the subgroup $C_{1,i}$ generated by all the $(b_{1,i} - \chi_i(b_{1,i}))\F_p[H]u_i$ as $b_{1,i}$ varies in $B_{1,i}$. Furthermore, we may repeat an analogous construction over $L'$ to construct a subgroup $C_{2,i}$. \\
As the action of $T_2$ on $\F_p[H]u_i / C_{2,i}$ is trivial, $C_{2,i}$ must correspond to a subfield of $M_i$ containing $NM_{2,i}$, which must also be an abelian $p$-extension of $L'$ where the operation of $S_{2,i}/T_2$ coincides with the multiplication by the values of $\chi_i'$. However, by definition $M_{2,i}$ is the maximal abelian $p$-extension of $L'$ with this property, thus $B_{2,i}=C_{2,i}$. \\ 
Let us then take the composite $\prod M_{1,i}$ of all the $M_{1,i}$, which is a subfield of $K_{m+1}$, and likewise $\prod M_{2,i}$ is a subfield of $L_{m+1}$. By construction, these two fields correspond to each other by $\sigma_{m+1}$ as they are composites of fields corresponding by $\sigma_{m+1}$. By Proposition \ref{Existence1} any element of the Galois group $G(M/\prod M_{1,i})$ is conjugate to an element of $G(M/\prod M_{2,i})$ by an element of $E$. Furthermore, $\F_p[H]^{n+1}$ is a normal subgroup of $E$. We may therefore consider the subgroups $A_1$ and $A_2$ of $\F_p[H]^{n+1}$, corresponding to the subfields $N \prod M_{1,i}$ and $N \prod M_{2,i}$ of $M$ respectively,  and we get that any element of $A_1$ must be conjugate to an element of $A_2$ by an element of $E$. Finally, since $C_{2,i}$ is the subgroup of $\F_p[H]^{n+1}$ corresponding to $NM_{2,i}$, we have that $A_2 = \sum_i C_{2,i}$. By the correspondence induced by $\sigma_{m+1}$, we get $A_1 \supseteq \sum_i C_{1,i}$. Furthermore, by the split exact sequence, conjugation corresponds to the action on $\F_p[H]^{n+1}$ given by left multiplication by an element of $H$. Therefore, we may fix an element
 \[a = \sum_{t_1 \in T_1} (t_1 - 1)u_0 + \sum_{i=1}^m (s_{1,i}-\chi_i(s_{1,i}))u_i \]
in $A$. For some $h \in H$, we get $ha \in A_2$ which we may rewrite as the two equations
\[h  \sum_{t_1 \in T_1} (t_1 - 1)u_0  \in B_{2,0}\] and \[h (s_{1,i}-\chi_i(s_{1,i}))u_i \in B_{2,i}.\] Expanding the first equation, we get \[ h \sum_{t_1 \in T_1} (t_1 - 1) \in \sum_{t_2 \in T_2} (t_2 - 1) \F_p[H]u_0\] and since \[\sum_{t_2 \in T_2} t_2 \sum_{t_2 \in T_2} (t_2 - 1) = 0 \in \F_p[H],\] we can rewrite this as follows: \[ \sum_{t_2 \in T_2} t_2 h \sum_{t_1 \in T_1} (t_1 - 1) = 0 \in \F_p[H]u_0. \]
We then fix an element $t'_1 \in T_1$. As the sum cancels out in $\F_p[H]u_0$, the coefficient of $h t'_1 \in H$ in the left side of the sum must be a multiple of $p$. The number of terms in the sum of the form $t''_2 h t''_1$ for some $t''_1 \in T_1$ and $t''_2 \in T_2$ must be less than $|H|^2$, which by our assumption is less than $p$. Then, the number of terms of the form $t''_2 h t''_1=h t'_1$ (which are all elements of $H$) is also less than $p$. \\ We then get that $h t'_1$ must cancel out with a term of the form $-t'_2 h$ for some $t'_2 \in T_2$, that is $t'_2 h = h t'_1$, and so we get $h^{-1} T_2 h \subseteq T_1$, therefore $h^{-1}$ induces an injective homomorphism $L' \to K'$. \\
Following the same idea as above, we observe $h (s_{1,i}-\chi_i(s_{1,i}))u_i \in B_{2,i}$ may be rewritten as 
\[\sum_{s \in S_{2,i}} s \chi_i'(s)^{-1} h (s_{1,i} - \chi_i(s_{1,i})) = 0 \in \F_p[H]u_i.\]
Then, repeating the same argument used above for $h t'_1$, we have the coefficient of $h s_{1,i}$ in the sum must be $0$, then for some $s' \in S_{2,i}$ we have $h s_{1,i} = s' h$ and $\chi_i'(s')=\chi_i(s_{1,i})$. 
Then, $h_{2,i} = s_{2,i} T_2 = s' T_2$ follows by the definition of $\chi_i'$. Furthermore, as $h^{-1}s'=s_{1,j}h^{-1}$, the actions defined by $h^{-1}\sigma(h_{1,j})$ and $h_{1,j} h^{-1}$ on $L'$ coincide. \\
As the $h_{1,i}$ generate $H_1$, it follows that $h^{-1}$ determines an embedding $L' \to K'$ which induces $\sigma$. However, $L'$ is a Galois subextension of $L_m/L$ chosen arbitrarily, so we may construct the set $\mathfrak{A}_{L'}$ of all embeddings $L' \inj K'$ for every finite Galois extension $L'$ of $L$ contained in $L_m$. These sets are non-empty and finite (as $h \in H$, and $H$ is finite). Then, the family $\mathfrak{A}_{L'}$ indexed over $L'$ defines a projective system of non-empty finite sets, so we may take the inverse limit over all $L'$ finite Galois subextensions of $L_m/L$. We then obtain that the set of embeddings $\tau_m: L_m \inj K_m$ inducing $\sigma_m$ is non-empty, as desired.
\end{proof}

We end this section by giving a result which is analogous to Theorem 1 in \cite{S-T}, and which is simply obtained as a consequence of Theorem \ref{Existence} when $m=1$.

\begin{corollary}\label{EmbeddingType} The following statements are true:
\begin{enumerate}
\item Let $\sigma_2: G_K^2 \to G_L^2$ be a homomorphism of profinite groups that satisfies condition $(\dagger^+)$. Then, there exists an embedding $\tau: L \to K$.
\item Let $\sigma_3: G_K^3 \to G_L^3$ be a homomorphism of profinite groups that satisfies condition $(\dagger)$. Then, there exists an embedding $\tau: L \to K$.
\end{enumerate}
\end{corollary}
\begin{proof}
Both statements follow naturally from Theorem \ref{Existence} when $m=1$, together with Remark \ref{EndowedMap} for the second assertion.
\end{proof}

\begin{remark}
It is known from results of Onabe \cite{Onabe}, which were later improved first by Angelakis and Stevenhagen \cite{A-S} and then by Gras \cite{Gras14} that if we were to replace $\sigma_2$ in the above corollary with a homomorphism $\sigma_1: G_K^{\ab} \to G_L^{\ab}$ then, independently of which condition we put on $\sigma_1$, we may not guarantee the existence of an embedding $\tau: L \to K$, as there are non-isomorphic number fields $K, \, L$ such that $G_K^{\ab} \isom G_L^{\ab}$. In this sense, the above corollary may only be improved by asking for a weaker condition.
\end{remark}

\section{Uniqueness in the m-step Hom-Form}

The next results correspond to similar results by Uchida for solvably closed extensions (\cite{UchidaHom}, Lemma 3 and the Corollary immediately following).

\begin{proposition}\label{Lemma3U} 
Let $K$ and $L$ be number fields and assume that $K$ and $L$ are contained in the same separable closure $\Omega$ of $\Q$. Let $m \ge 1$ be an integer, and consider their $m$-step solvably closed extensions $K_{m}$ and $L_{m}$ contained in $\Omega$.
Then, we have that if $K_{m-1}$ is not contained in $L_m$, the composite $K_{m} L_{m}$ is an infinite extension of $L_m$.
\end{proposition}

\begin{proof}
Assume $K_{m-1}$ is not contained in $L_m$. Then, there must be a finite extension $K'$ of $K$ contained in $K_{m-1}$ such that $K'$ is not contained in $L_m$. We may take a finite Galois extension $M/\Q$ such that $M$ contains both $K'$ and $L$, and let us denote $H=G(M/\Q)$ and $H_1=G(M/K')$. By (\cite{NWS}, Proposition 9.2.9), if we let $p$ be a prime number such that $p$ does not divide $|H|$, and consider the split group extension 
\[1 \to \F_p[H] \to E \to H \to 1\] there exists a Galois extension $M'$ of $\Q$ containing $M$ such that $G(M'/\Q)$ is isomorphic to $E$ and $G(M'/M)$ is isomorphic to $\F_p[H]$. \\
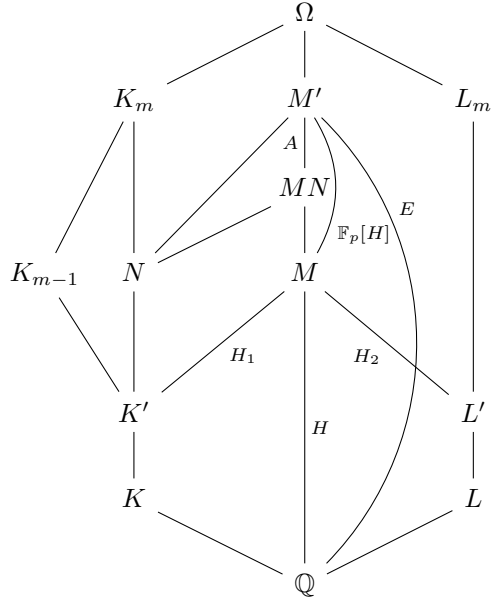
\begin{center}
\begin{tikzcd}
&[-2 em] &&[-1em] \Omega \arrow[dash]{dll} \arrow[dash]{d} \arrow[dash]{drr}&[-1em]&\\
&K_m \arrow[dash]{ddl} \arrow[dash]{dd}&& M' \arrow[dash, near end, bend left=30]{dd}{\F_p[H]} \arrow[dash]{ddll} \arrow[dash,bend left=45, near start]{dddddd}{E} \arrow[dash,swap]{d}{A} && L_m \arrow[dash]{dddd}\\
&&&MN \arrow[dash]{dll}\arrow[dash]{d}&&\\
K_{m-1} \arrow[dash]{ddr} &N\arrow[dash]{dd}&&M\arrow[dash]{ddll}{H_1} \arrow[dash,swap]{ddrr}{H_2} \arrow[dash]{dddd}{H}&&\\
\\
&K'\arrow[dash]{d}&&&&L'\arrow[dash]{d}\\
&K\arrow[dash]{drr} &&&&L\arrow[dash]{dll}\\
&&& \Q &&\\
\end{tikzcd}
\end{center}
Let $N$ be the maximal abelian $p$-extension of $K$ contained in $M'$, and observe $N$ and $M$ are linearly disjoint extensions of $K'$ as $p$ does not divide $|H|$. \\ Consider the composite $MN \subseteq  M'$. 
Then, $MN$ coincides with the maximal abelian $p$-extension $M^*$ of $M$ contained in $M'$ (which can be obtained as a quotient of $\F_p[H]$) such that the action of $H_1=G(M/K')$ on $G(M^*/M)$ is trivial. We may also observe that $MN$ corresponds to the subgroup $A$ of $\F_p[H]$ determined by all the elements where the action of $H_1$ is not trivial, that is \[G(M'/MN) = A = \sum_{h_1 \in H_1} (h_1-1)\F_p[H].\]
Let us then consider the field $L' = K'L \cap L_m$. Since $K'$ is not contained in $L_m$, it follows that it is also not contained in $L'$. If we consider the subgroup $H_2$ of $H$ corresponding to $L'$, it is not contained in $H_1$, and if we construct the subgroup $A' =\sum_{h_2 \in H_2} (h_2-1)\F_p[H]$ we have that $A'$ is not contained in $A$ and so the action of $H_2$ on $\F_p[H] / A$ is not trivial, and so there is no abelian extension $N'$ of $L'$ such that $N'M=NM$. \\
Consider the Galois group $G(K' L_m /K' L)$, which is isomorphic to $G(L_m/L')$. Since $N$ is an abelian extension of $K' \subseteq K_{m-1}$, we have $N \subseteq K_m$, and immediately we have $N L_m$ is an extension of $L_m$ contained in $K_m L_m$. If we assume that $N L_m$ is contained in $K' L_m$, then $NL$ is a subfield of $NL_m$, and since $N/K'$ is abelian the extension $NL/K'L$ is also abelian. Therefore, by the isomorphism of Galois groups above we can find an abelian extension $N'$ of $L'$ such that $N L = N' K'$. However, this means that $MN$ coincides with $MN'$, which is a composition of $N'$, an abelian extension of $L'$, and $M$. We then get a contradiction, and so $NL_m$ is a non-trivial abelian $p$-extension of $K' L_m$ contained in $K_m L_m$. \\
It now follows that $K_m L_m$ contains an extension of $L_m$ whose degree is a multiple of $p$, and repeating this argument for the infinitely many primes $p$ not dividing $|H|$ we get that $K_m L_m$ must necessarily be an infinite extension of $L_m$.
\end{proof}

\begin{corollary}\label{Cor3U}
Let $m \ge 1$ be a positive integer, $K$ and $L$ be number fields contained in a same separable closure $\Omega$ of $\Q$, and consider their respective maximal $m$-step solvable extension $K_m$ and $L_m$ contained in $\Omega$. Assume that there exists a number field $M$ such that $M K_m = M L_m$. Then, $K_{m-1}$ is contained in $L_m$. In particular, $K \subseteq L_m$.
\end{corollary}
\begin{proof}
Since $M K_m = M L_m \supseteq K_m$, we get that $K_m L_m$ is contained in $M L_m$, and is therefore a finite extension of $L_m$. By Proposition \ref{Lemma3U}, we get that this must necessarily mean $K_{m-1}$ is contained in $L_m$. \\
The second assertion follows immediately as $K \subseteq K_{m-1}$
\end{proof}

In the rest of this section, we will consider a homomorphism of profinite groups \[\sigma_m: G_K^m \to G_L^m\] such that $H = \sigma_m(G_K^m)$ is an open subgroup of $G_L^m$, and that there exists an embedding of fields \[\tau: L_m \inj K_m \] such that $ \tau \sigma_m(g) = g \tau$ for all $g \in G_K^m$. \\
We will denote the finite extension of $L$ contained in $L_m$ corresponding to $H$ by $\widetilde{L}$. We will also denote the subfield of $K_m$ corresponding to the kernel of $\sigma_m$ by $\Lambda_m$. For all $1 \le i \le m-1$, we may consider the homomorphisms $\sigma_{i}: G_K^i \to G_L^i$ induced as in Proposition \ref{Diagram}, we will denote the field corresponding to $\ker(\sigma_i)$ by $\Lambda_{i}$. \\
Observe that, by definition $\tau(\widetilde{L})$ is fixed under the action of all $g \in G_K^m$, and so $\tau(\widetilde{L}) \subseteq K$. It is also the maximal extension of $\tau(L)$ contained in $\tau(L_m)$ with this property, so it follows that $\tau(L_m) \cap K = \tau(\widetilde{L})$, and $K \tau(L_m) = \Lambda_m$. \\
Finally, observe that as $\tau$ gives an isomorphism between $L_m/L$ and $\tau(L_m)/\tau(L)$, we have $\tau(L_m) = \tau(L)_m$. \\
In the following statement, we will require the same property we needed to show Proposition \ref{isopen}, namely $\widetilde{L} \subseteq L_{m-1}$

\begin{theorem}\label{Unique}
Let $m \ge 2$ be a positive integer, $\sigma_m: G_K^m \to G_L^m$ a homomorphism of profinite groups, and assume that $\sigma_m$ has open image and $\sigma_m(G_K^m)$ contains $G_L^m[m-1]$. Then, if $\tau$ and $\rho$ are homomorphisms $\tau, \rho: L_m \inj K_m$ such that $\tau \sigma_m(g) = g \tau$ and $\rho \sigma_m(g) = g \rho$, we have $\tau=\rho$, that is, the property $\tau \sigma_m(g) = g \tau$ determines $\tau$ uniquely.
\end{theorem}
\begin{proof}
First, observe that $\tau(L_m)$ and $\rho(L_m)$ are both contained in $K_m$. \\
Since $K \tau(L_m) = \Lambda_m = K \rho(L_m)$, by Corollary \ref{Cor3U} we have $\tau(L_{m-1}) \subseteq \rho(L_m)$. Since by assumption $\widetilde{L} \subset L_{m-1}$, we get $\tau(\widetilde{L}) \subset \tau(L_{m-1}) \subset \rho(L_m)$. By construction $\tau(\widetilde{L}) \subseteq K$, and furthermore $\tau(L_{m-1}) \cap K = \tau(\widetilde{L})$. From the above inclusion we then obtain $\tau(\widetilde{L}) \subseteq \rho(\widetilde{L}) = \rho(L_m) \cap K$. However, since $\tau(\widetilde{L})$ and $\rho(\widetilde{L})$ have the same degree, it follows they are equal. \\
Now, observe that $K$ and $\rho(L_m)$ are disjoint extensions of $\rho(\widetilde{L})$, and so the natural isomorphism $G(\Lambda_m/K) \isom G(\rho(L_m)/\rho(\widetilde{L}))$ gives us $K \rho(L_{m-1}) \cap \rho(L_m) = \rho(L_{m-1})$. \\ However, $K \rho(L_{m-1}) = \Lambda_{m-1} = K \tau(L_{m-1})$, and as $K \rho(L_{m-1}) \cap \rho(L_m) \supseteq \tau(L_{m-1})$, we get $\rho(L_{m-1}) \supseteq \tau(L_{m-1})$. We may also switch $\tau$ and $\rho$ to obtain an equality. \\
It now follows that $\rho(L_m) = \tau(L_m)$ as well, as they are maximal abelian extensions of the same field in the same separable closure, and so the composition $\tau \rho^{-1}$ is an automorphism of $\tau(L_m)$. \\ 
Let us then take an element $g \in G_K^m$. By compatibility property in the statement, we get $\tau \rho^{-1} g = \tau \sigma_m(g) \rho^{-1} = g \tau \rho^{-1}$, that is $\tau \rho^{-1}$ commutes with every element $g \in G_K^m$. Then $\tau \rho^{-1}$ also commutes with the image of $g$ by the surjective homomorphism $G_K^m \surj G(\tau(L_m)/\tau(\widetilde{L}))$ for all $g \in G_K^m$, which implies $\tau \rho^{-1}$ centralises $G(\tau(L_m)/\tau(\widetilde{L}))$. However, by (\cite{S-T}, Corollary 1.7), we finally get $\tau \rho^{-1}$ must be the identity, and so $\tau = \rho$ as desired.
\end{proof}

An immediate case where the condition we ask for in the theorem above is satisfied is when $\sigma_m$ is surjective. 

\begin{proposition}\label{Large}
Let $K,L$ be number fields. Then, there exists an integer $m' \ge 0$ such that for all integers $m > m'$ the following holds: \\ Assume we have a homomorphism of profinite groups $\sigma_m: G_K^m \to G_L^m$ with open image and an embedding $\tau: L_m \inj K_m$ an embedding of fields with the property $\tau \sigma_m(g) = g \tau$, then $\tau$ is uniquely determined by the latter property. 
\end{proposition}
\begin{proof}
Let us assume without loss of generality that there exists an embedding $\tau': L \inj K$, and let $\widehat{K}$ be the Galois closure of $K$ over $\Q$. Observe that we have a finite Galois extension $\widehat{K}/\tau'(L)$, and set $G=G(\widehat{K}/\tau'(L))$. Then, we may consider the maximal solvable quotient $G_{\sol}$ of $G$, which correspond to an extension $M'/\tau'(L)$. By construction, $M'$ is the maximal solvable extension of $\tau'(L)$ contained in $\widehat{K}$. \\
Observe now that since $M'/\tau'(L)$, is a solvable extension, there exists an integer $m'$ such that $M' \subset \tau'(L)_{m'}$ where $\tau'(L)_{m'}$ is the maximal $m'$-step solvable extension of $\tau'(L)$ contained in $K_{m'}$. By maximality of $G^{\sol}$, we have $M' = \widehat{K} \cap \tau'(L)_{m'}$, and so $M' \supseteq K \cap \tau'(L)_{m'}$. \\
We now apply the above construction to the following case: let us take any integer $m \ge m'$, and a homomorphism of profinite groups $G_K^m \to G_L^m$ such that there exists an embedding $\tau: L_m \to K_m$ compatible with $\sigma_m$ as in the statement. Denote by $\widetilde{L}$ the subfield of $L_m/L$ corresponding to the image. By construction, $\tau(\widetilde{L}) = \tau(L_m) \cap K$, furthermore, since $m' < m$, we may say $\tau(L)_{m'}$ coincides with $\tau(L_{m'})$. \\ Observe that $\tau(\widetilde{L})$ is a solvable extension of $\tau(L)$ contained in $K$, so if we construct the extension $M$ of $\tau(L)$ analogously to the construction of the extension $M'/\tau'(L)$ above, by maximality $M \supseteq \tau(\widetilde{L})$, and we have \[\tau(L_m') \supset M \supset \tau(\widetilde{L}).\]
Since $m > m'$, we get indeed that $\widetilde{L} \subset L_{m'} \subset L_{m-1}$ and we are in the conditions to apply Theorem \ref{Unique}. The statement now follows.
\end{proof}

\begin{remark}\label{ForceLarge}
Note that in the above proof $m'$ depends on $K$ and $L$. Using the notations of Proposition \ref{Large} and its proof, it is possible to give an estimate of an upper bound for $m'$ independently of $\sigma_m$ and $\tau$ as follows: \\ Let $n = [\widehat{K}:\Q]$, and $n' = [L:\Q]$. Then, we have immediately $|G|=n/n'$, and $|G^{\sol}| \le |G|$. However, as $G^{\sol}$ is solvable and every non-trivial abelian extension has degree at least $2$, we can give an upper bound for its derived length by $\lfloor \log_2(|G^{\sol}|) \rfloor$, which by the above discussion is $\le \lfloor n/n' \rfloor$. Therefore, $m' = \lfloor \log_2([\widehat{K}:\Q]/[L:\Q]) \rfloor$ is a valid upper bound for which $m > m'$ guarantees uniqueness. \end{remark}

\section{A conditonal m-Step Hom Form}

In this brief section, we state and prove our main result, which consists in a formulation of a conditional $m$-Step Hom-Form for Number Fields.

\begin{theorem}\label{Main}
Let $K,L$ be number fields, $m \ge 2$ an integer, and $\sigma_{m+3}: G_K^{m+3} \to G_L^{m+3}$ a homomorphism of profinite groups. Consider the homomorphism of profinite groups $\sigma_m: G_K^m \to G_L^m$ induced by $\sigma_{m+3}$. Assume that $\sigma_m$ satisfies condition $(\dagger)$ and that $\sigma_{m+3}(G_K^{m+3}) \supseteq G_L^{m+3}[m-1]$. \\  Then, $\sigma_m(G_K^m)$ is open in $G_L^m$, and there exists a unique embedding of fields $\tau: L_m \inj K_m$ such that for all $g \in G_K^m$ we have $\tau \sigma_m(g) = g \tau$. Furthermore, $\tau$ restricts to an embedding $L \inj K$.
\end{theorem}
\begin{proof}
Observe first that since $\sigma_m$ satisfies condition $(\dagger)$, the homomorphism $\sigma_{m+2}: G_K^{m+2} \to G_L^{m+2}$ determined by $\sigma_{m+3}$ satisfies condition $(\dagger)$ by Proposition \ref{starllift}, and $\sigma_{m+1}$ satisfies condition $(\dagger^+)$. The existence of $\tau_m$ now follows from Theorem \ref{Existence}. \\
By Proposition \ref{isopen} it also follows $\sigma_m$ has open image (as the projection are continuous, it also follows that $\sigma_{m+3}$ has open image). We are then in the conditions to apply Theorem \ref{Unique} as well, that is $\tau_m$ is uniquely determined. The statement now follows immediately.
\end{proof}

We can also give the following variation. The proof of the following Theorem is very similar to the one of Theorem \ref{Main}, with the inclusion of Proposition \ref{Large} and the argument in Remark \ref{ForceLarge}, so we will omit it.

\begin{theorem}\label{Variation}
Let $K,L$ be number fields, $\widehat{K}$ a Galois closure over $\Q$ of $K$, and $m \ge \max\{2;\lfloor log_2(\frac{[\widehat{K}:\Q]}{[L:\Q]}) \rfloor + 1\}$ an integer. Let $\sigma_{m+3}: G_K^{m+3} \to G_L^{m+3}$ be a homomorphism of profinite groups. Consider the homomorphism of profinite groups $\sigma_m: G_K^m \to G_L^m$ induced by $\sigma_{m+3}$ and assume that $\sigma_m$ satisfies condition $(\dagger)$. Then, $\sigma_m(G_K^m)$ is open in $G_L^m$, and there exists a unique embedding of fields $\tau: L_m \inj K_m$ such that for all $g \in G_K^m$ we have $\tau \sigma_m(g) = g \tau$. Furthermore, $\tau$ restricts to an embedding $L \inj K$.
\end{theorem}

\begin{remark}
The conditions we ask for in Theorems \ref{Main} are group-theoretic, as $(\star_l)$-subgroups can be recovered without looking at decomposition groups as in Definition \ref{starldef}, and the condition $\sigma_{m+3}(G_K^{m+3}) \supseteq G_L^{m+3}[m-1]$ is also group-theoretic. 
\end{remark}

\begin{remark}\label{Sharp}
Let us take a homomorphism $\sigma: G_K \to G_L$ satisfying the condition Uchida asks for in his conditional result (\cite{UchidaHom}, Theorem 2), namely that for every prime $\p \in \Primes_{\K}$ with decomposition group $D_{\p} \subset G_K^m$, there exists a unique prime $\q \in \Primes_{\Lb}$ so that $\sigma(D_{\p})$ is an open subgroup of $D_{\q}$, and let us consider the homomorphism $\sigma_m: G_K^m \to G_L^m$ induced by $\sigma$ as in Proposition \ref{Diagram}. \\
Lemma 4 in \cite{UchidaHom} shows that primes $\p$ and $\q$ as above have the same residue characteristic $p$. Taking $l \neq p$, and appropriate $l$-Sylow subgroups $D_{\p,l}$ and $D_{\q,l}$, the condition that $\sigma(D_{\p})$ is an open subgroup of $D_{\q}$ implies $\sigma(D_{\p,l})$ is an open subgroup of $D_{\q,l}$, and so it is necessarily mapped injectively by $\sigma$. Furthermore, \cite{S-T}, Proposition 1.1.(vii) shows that $D_{\p,l}$ and $D_{\q,l}$ are mapped injectively by the quotient maps $G_K \surj G_K^m$ and $G_L \surj G_L^m$. Commutativity (cf. Proposition \ref{Diagram}), together with Lemma \ref{zlinj}, allows us to conclude that indeed the image of $D_{\p,l}$ by the quotient map $D_{\p} \surj D_{\p}^m$ is mapped injectively to the image of $D_{\q,l}$ by the quotient map $D_{\q} \surj D_{\q}^m$. \\ 
As every non-archimedean prime of $K_m$ is the restriction of some prime of $\overline{K}$, repeating the argument above for every prime of $\overline{K}$ and every prime number $l$ we get $\sigma_m$ satisfies condition $(\dagger)$. We may observe that Uchida's condition additionally requires that $\sigma(D_{\p,p})$ is open in $D_{\q,p}$, which is not needed in our result. \\
We may then say that Theorem \ref{Variation} is a proper sharpening of Uchida's result, as any homomorphism satisfying Uchida's condition induces a homomorphism satisfying condition $(\dagger)$ at some finite step $m$.

\end{remark}
\section{The m-step Hom-Form over the rational numbers}

In this section, for $K=\Q$, $L$ a number field and an integer $m \ge 2$, we fix a homomorphism of profinite groups $\sigma_{m+2}: G_K^{m+2} \to G_L^{m+2}$ such that the image of $G_K^{m+2}$ is open in $G_L^{m+2}$. We are going to work mainly with the induced homomorphism $\sigma_m: G_K^m \to G_L^m$, which also has open image by Proposition \ref{Diagram}.

\begin{proposition}\label{SurjectiveQ} 
The following hold: $L=\Q$, and $\sigma_m$ is surjective.
\end{proposition}
\begin{proof}
Let $\widetilde{L}$ be the subfield of $L_m/L$ corresponding to the image of $\sigma_m$. Then, as $\widetilde{L}/L$ is finite, there must exist a quadratic extension $L'$ of $\widetilde{L}$ contained in $L_m$ such that $L'$ is totally imaginary, and let $K'$ be the quadratic extension of $K$ corresponding to $L'$ by $\sigma_{m}$. \\ It then follows that $\sigma_{m+1}(G(K_{m+1}/K')) \subseteq G(L_{m+1}/L')$, and since the image of $\sigma_{m+1}$ is open in $G_L^{m+1}$, the image of $G(K_{m+1}/K')$ is open in $G(L_{m+1}/L')$. Since $L' \subset L_m$, it follows $L'^{\ab} \subseteq L_{m+1}$, and we may take the canonical quotient $G(L_{m+1}/L') \surj G_{L'}^{\ab}$ and by composition obtain a map $G(K_{m+1}/K') \to G_{L'}^{\ab}$, which has open image. Furthermore, this map must factor through $G_{K'}^{\ab}$, therefore we get a homomorphism of profinite groups $\iota: G_{K'}^{\ab} \to G_{L'}^{\ab}$ with open image which is compatible with $\sigma_{m+1}$. \\
Let $s$ be the $\Z_p$-rank of $L'$, and let $L''$ be the unique $\Z_p^s$-extension of $L'$. Then, if we take the composition of $\iota$ with the quotient $G_{L'}^{\ab} \surj G(L''/L')$, the image is again open and so isomorphic to $\Z_p^s$ as well, which also means it corresponds to a $\Z_p^s$-extension of $K'$. \\
Since $L'$ is totally imaginary, $s \ge [\widetilde{L}:\Q]+1$. However, since $K'$ is a quadratic extension of $\Q$, its $\Z_p$-rank is $\le 2$. Therefore, $s \le 2$, which implies $[\widetilde{L}:\Q]=1$. This not only gives $\sigma_m$ is surjective, but that $L=\Q$ as well.
\end{proof}

\begin{corollary}\label{ZpQ}
Let $L'$ be the unique $\Z_2$-extension of $L$. Then, the extension $K'$ of $K$ corresponding to it by $\sigma_m$ is the unique $\Z_2$-extension of $K$.
\end{corollary}
\begin{proof}
As $\sigma_m$ is surjective by the previous proposition, by taking quotients, we can induce from $\sigma_m$ an isomorphism $\sigma: G(K'/K) \isomto G(L'/L) \isom \Z_2$. The statement now follows immediately.
\end{proof}

As a consequence of the above corollary, we have that the extension $L(\sqrt{2})$ corresponds to $K(\sqrt{2})$ by $\sigma_m$. We will use this fact in the next proposition, which encompasses the equivalent of Lemmas 1 and 2 in \cite{UchidaHom}.

\begin{proposition}\label{DecMapQ}
Let $S$ be the set of all prime numbers $p$ such that for each decomposition group $D_p$ above $p$ in $G_K^m$ (which is determined up to conjugation) we have $\sigma_m(D_p) \subseteq D'_p$, where $D'_p$ is a decomposition group above $p$ in $G_L^m$. Then, the set $S$ cofinite in $\Primes_{\Q}$.
\end{proposition}
\begin{proof}
Let us denote the $\Z_2$-extensions of $K$ and $L$ respectively by $K_\circ$ and $L_{\circ}$, taking quotients gives us an isomorphism $\sigma_{\circ}: G(K_{\circ}/K) \to G(L_{\circ}/L)$ induced from $\sigma_m$ (see Remark \ref{quotients}). \\
Observe that the odd prime number $p$ has infinite decomposition group in the unique $\Z_2$-extension of $\Q$, and so by Corollary \ref{ZpQ} we get a quotient of $D_p$ (isomorphic to $\Z_2$) in this $G(K_\circ/K)$, which we denote $\overline{D}$. Furthermore, this is a pro-$2$-group, so a $2$-Sylow $D_{p,2}$ maps surjectively to $\overline{D}$. It follows by commutativity of the diagram 
\begin{center}
\begin{tikzcd}
G_K^m \arrow["\sigma_m"]{r} \arrow[d, two heads] & G_L^m \arrow[d, two heads] \\
G(K_{\circ}/K) \arrow["\sigma_{\circ}"]{r} & G(L_{\circ}/L) 
\end{tikzcd}
\end{center} that as $\bar{D}$ is mapped isomorphically to $G(L_{\circ}/L)$ then $\sigma_m(D_{p,2})$ needs to also be infinite. \\
Let us now consider the extension $L(\sqrt{p})/L$, and the quadratic extension $K'$ of $K$ corresponding to it by $\sigma_m$. Observe that there may only be one quadratic extension of $L$ corresponding to $K(i)$ (resp. $K(\sqrt{-2})$) by $\sigma_m$, and in particular there may only be one odd prime $p_1$ (resp. $p_2$) such that $L(\sqrt{p_1})$ (resp. $L(\sqrt{p_2})$) corresponds to $K(i)$ (resp. $K(\sqrt{-2})$) by $\sigma_m$. Observe that there may also not be primes $p_1$ or $p_2$ for which this is true.\\ Except when $p = p_1$ or $p = p_2$, (we already know that $K(\sqrt{2})$ corresponds to $L(\sqrt{2})$ so we do not need to exclude it) there is then an odd prime $q$ that ramifies in the extension $K'/K$ corresponding to $L(\sqrt{p})/L$ by $\sigma_m$. By this argument the set of all the prime numbers $q \in \Primes_{\Q}$ for which this is true is then a subset of density $1$ in $\Primes_{\Q}$, which we'll denote $T$. \\
In particular, the inertia subgroup of a $2$-Sylow subgroup of a decomposition group $D_{q} \subset G_K^m$ has a non-trivial quotient which maps injectively to $G(K'/K)$ and is mapped isomorphically to $G(L'/L)$ by the map $\sigma: G(K'/K) \to G(L'/L)$ induced from $\sigma_m$ by quotients as in Remark \ref{quotients}. Then, the $2$-Sylow subgroup of $D_q$, which we denote by $D_{q,2}$, must have infinite image by $\sigma_m$, and the image of its inertia subgroup must be non-trivial. \\
As $I_{q,2}$ is isomorphic to $\Z_2$, and its image by $\sigma_m$ is non-trivial by the above argument, we get that either $I_{q,2}$ is mapped injectively by $\sigma_m$, and so is $D_{\q,2}$, or $\sigma_m(I_{q,2})$ contains a $2$-torsion element. Let us take a decomposition group at $q$ in $G_K^{m+1}$ that maps surjectively to $D_{\q}$, which we will denote $\D_{q}$. Then, any $2$-Sylow subgroup $\D_{q,2}$ of $\D_q$ which maps surjectively to $D_{q,2}$ may not be mapped injectively by $\sigma_{m+1}$ by Proposition \ref{starllift} (else $D_{q,2}$ would also be mapped injectively), and since $\sigma_{m+1}(\D_{\q,2})$ must have a quotient isomorphic to an extension of $\Z_2$ by $\Z/2\Z$ and is itself a quotient of a $2$-decomposition-like group. Furthermore, a non=trivial torsion element in $G_L^{m+1}$ is order $2$ by Proposition \ref{TamagawaTorsion}, so $\sigma_{m+1}(\D_{\q,2})$ is also isomorphic to an extension of $\Z_2$ by $\Z/2\Z$. It then follows the $2$-torsion in $\sigma_m(I_{q,2})$ is the image of torsion elements of $G_L^{m+1}$, and by Proposition \ref{TamagawaTorsion}, this means they correspond to the decomposition group of an archimedean prime. However, these non-trivial torsion element in $\sigma_m(I_{q,2})$ which correspond to complex conjugation must map non-trivially to $G(L({\sqrt{p}})/L)$, which is a contradiction. It then follows that $D_{q,2}$ is mapped injectively by $\sigma_m$. \\
By Proposition \ref{starllift}, it also follows that $\D_p$ maps surjectively to $D_{\p}$ by the natural quotient, an a $2$-Sylow subgroup $\widetilde{D}_{p,2}$ which maps surjectively to $D_{p,2}$, we have by $\widetilde{D}_{\p,2}$ is mapped injectively by $\sigma_{m+1}$, and so by Proposition \ref{MapStarl} it is mapped to a subgroup of $G_L^{m+1}$ satisfying property $(\star_2)$. However, since $p$ is the only odd prime ramifying in $L(\sqrt{p})/L$, it follows that $\sigma_m(D_{q,2}) \subset D'_p$ for some decomposition group $D'_p$ above $p$ in $G_L^{m}$ by Proposition \ref{DecMap}, and $\sigma_m(I_{q,2}) \subseteq I'_{p,2}$. \\
We now only need to show that $q = p$. Let us fix an odd prime number $p$, and $q$ so that $\sigma(D_{q}) \subseteq D_{p}$, and $n \ge 3$ an integer. Consider a primitive $2^n$-th root of unity $\zeta_{2^n}$, and the extension $L(\zeta_{2^n})/L$. We also take the Galois extension $K''$ of $K$ corresponding to  $L(\zeta_{2^n})$ by $\sigma_m$. Then, $\sigma_m$ induces an isomorphism $G(K''/K) \to G(L(\zeta_{2^n})/L)$. As the odd prime $p$ does not ramify in $L(\zeta_{2^n})/L$, $q$ does not ramify in $K''/K$. \\
Let us denote by $L(\zeta_{2^{\infty}})$ the composite of $L(\zeta_{2^n})$ for all positive integers $n$, and by $\widetilde{K}$ the subextension of $K_m/K$ corresponding to $L(\zeta_{2^{\infty}})$ by $\sigma_m$. Observe that we naturally have an isomorphism $G(\widetilde{K}/K) \isom G(L(\zeta_{2^{\infty}})/L) \isom \Z_2 \times \Z/2\Z$ induced from $\sigma_m$, which implies that the unique $\Z_2$-extension of $K$ is a subfield of $\widetilde{K}$. Then, by assuming $n$ is large enough, we may also assume that $q$ is not totally split in $K''/K$. \\ Let $s$ be the $\Z_q$-rank of $L(\zeta_{2^n})$ (which is naturally $\le$ to the $\Z_q$-rank of $K''$). We may take the unique $\Z_q^s$-extension of $L(\zeta_{2^n})$ contained in $L_m$, which we will denote $\widehat{L}$, and the subextension $\widehat{K}$ of $K_m/K$ corresponding to it by $\sigma_m$. The construction gives us $\widehat{K}/K''$ is a $\Z_q^s$-extension. \\
Since $L(\zeta_{2^n})/L$ is a totally imaginary extension (and $L$ is, in fact, $\Q$), $s = [L(\zeta_{2^n}):L]/2 + 1 = [K'':K]/2 + 1$ (in general, the first equality has to be replaced with $\ge$, however since $L(\zeta_{2^n})$ is an abelian extension of $\Q$, Leopoldt's conjecture holds for any prime number in $L(\zeta_{2^n})$, and we have an equality in our case). As there are at most $[K'':K]/2$ distinct prime divisors of $q$ in $K''$ by our assumption on $n$, it follows that $s$ is greater than the number of prime divisors of $q$ in $K(\zeta_{2^n})$, which for $j \le s/2$ we will denote $\p_1,\p_2,...,\p_j$.\\
Let us take a decomposition group $D_{q} \subset G_K^m$ at $q$ (unique up to conjugation), and let $i$ be the integer for which $D_q \cap G(K_m/K'')$ is a decomposition group at $\p_i$, which we will denote $D_{\p_i}$. Then, let $I_{\p_i} = I_q \cap D_{\p_i}$. \\
The image of $I_{\p_i}$ by the natural quotient $G(K_m/K'') \surj G(\widehat{K}/K'') \isom \Z_q^s$ is independent from the choice of a conjugate determining $\p_i$ as the image is an abelian group. Let us assume that all of the images of the $I_{\p_i}$ with $i$ varying are isomorphic to $\Z_q$. Then, there exists a quotient of $G(\widehat{K}/K'')$ isomorphic to $\Z_q$ in which all the $I_{\p_i}$ map trivially. However, this is impossible as at least one prime above $q$ must ramify in any $\Z_q$-extension of a number field (cf. \cite{NWS}, Proposition 11.1.1). We may then fix $i$ so that if we take the prime $\p_{i}$, the image of $I_{\p_i}$ in $G(\widehat{K}/K'')$ contains a subgroup isomorphic to $\Z_q^2$. \\
The isomorphism given by $\sigma_m$ between $G(\widehat{K}/K'')$ and $G(\widehat{L}/L(\zeta_{2^n}))$ together with the fact that (for some appropriate choice of conjugates) $\sigma_m(D_{q}) \subseteq D'_p$ as we have proven above, shows that $I'_p$ contains a subgroup isomorphic to $\Z_q^2$, however this is only possible if $p = q$ as desired. \\
Therefore, the set $S$ in our statement coincides with the set $T$, which is cofinite (and so of density $1$) in $\Primes_{\Q}$, as desired.
\end{proof}

\begin{proposition}\label{T1Isom}
$\sigma_m$ has trivial kernel, and so it is an isomorphism
\end{proposition}
\begin{proof}
Let us denote by $\Lambda$ the subfield of $K_m$ corresponding to the kernel of $\sigma_m$.
Let $L'$ be a finite Galois extension of $L$ contained in $L_m$, and let $K'$ be the finite Galois extension of $K$ contained in $\Lambda$ corresponding to it by $\sigma_m$. Observe first that $K'$ and $L'$ have the same degree over $\Q$. \\ We may take an odd prime $p \in S$ that splits completely in $L'/L$. Observe that the image of $D_p$ by the map $G_K^m \surj G(K'/K)$ must be trivial by Proposition \ref{DecMapQ}, that is $p$ splits completely in $K'$. It then follows there is an embedding $K' \inj L'$ by Proposition 13.9 in \cite{NeukirchANT}, however equality between degrees shows this is in fact an isomorphism, we may also construct the set $\mathfrak{A}_{L'}$ of such isomorphisms, which is finite.\\
The inverse limit over all finite Galois subextensions $L'$ of $L_m/L$ of the family of sets $\mathfrak{A}_{L'}$ is non-empty as  all the $\mathfrak{A}_{L'}$ are finite and non-empty. Then, this shows that there is an isomorphism $\alpha: \Lambda \isomto L_m$. \\
Assume then $\Lambda \subsetneq K_m$. Then, $K_m/\Lambda$ is a non-trivial extension and $K_m$ is Galois over $\Q$. Thus, we may define using $\alpha$ a non-trivial extension $\widehat{L}$ of $L$ such that $G(K_m/K) \isomto G(\widehat{L}/L)$. However, the extension $\widehat{L}/\Q$ must be $m$-step solvable as it is isomorphic to $K_m/\Q$. By maximality, we get $L_m=\widehat{L}$ and so $\Lambda=K_m$. 
\end{proof}

\begin{corollary}\label{T1starl}
For all prime numbers $l$, the subgroups satisfying property $(\star_l)$ in $G_K^{m+1}$ are mapped to subgroups satisfying  property $(\star_l)$ by $\sigma_{m+1}$, that is $\sigma_{m+1}$ satisfies condition $(\dagger_l)$. \\ In particular, $\sigma_{m+1}$ satisfies condition $(\dagger)$ and $\sigma_m$ satisfies condition $(\dagger^+)$.
\end{corollary}
\begin{proof}
As $\sigma_m$ is an isomorphism, it follows immediately by Proposition \ref{T1Isom} that all the subgroups satisfying property $(\star_l)$ in $G_K^m$ are mapped injectively. \\ By Proposition \ref{starllift}, it then follows that subgroups satisfying property $(\star_l)$ in $G_K^{m+1}$ are also mapped injectively by $\sigma_{m+1}$, and we may then conclude using Proposition \ref{MapStarl} as $\sigma_{m+1}$ is induced by $\sigma_{m+2}$.
\end{proof}

The above corollary shows that if $m \ge 2$ we are in the conditions to apply our result on the Hom-Form, and we then get the following

\begin{theorem}\label{MainQ}
Let $m \ge 2$ be an integer, and let $\sigma_{m+3}: G_\Q^{m+3} \to G_L^{m+3}$ a homomorphism of profinite groups with open image. Then, $L=\Q$. Furthermore, the homomorphism $\sigma_m: G_\Q^m \to G_\Q^m$ induced by $\sigma_{m+3}$ is an isomorphism, and there exists a unique isomorphism of fields $\tau: \Q_{(m)} \to \Q_{(m)}$ such that $\sigma_m$ is determined by $\sigma_m(g) = \tau g \tau^{-1}$ for all $g \in G_K^m$.
\end{theorem}

\providecommand{\bysame}{\leavevmode\hbox to3em{\hrulefill}\thinspace}
\providecommand{\MR}{\relax\ifhmode\unskip\space\fi MR }
\providecommand{\MRhref}[2]{
  \href{http://www.ams.org/mathscinet-getitem?mr=#1}{#2}
}
\providecommand{\href}[2]{#2}


\begin{thebibliography}{10}

\bibitem{A-S}
Athanasios Angelakis and Peter Stevenhagen, \emph{Imaginary quadratic fields
  with isomorphic abelian {G}alois groups}, A{NTS} {X}---{P}roceedings of the
  {T}enth {A}lgorithmic {N}umber {T}heory {S}ymposium, Open Book Ser., vol.~1,
  Math. Sci. Publ., Berkeley, CA, 2013, pp.~21--39. \MR{3207406}

\bibitem{Gras14}
Georges Gras, \emph{On the structure of the galois group of the abelian closure
  of a number field}, Journal de Théorie des Nombres de Bordeaux \textbf{26}
  (2014), no.~3, 635--654.

\bibitem{Onabe}
Onabe Midori, \emph{On the isomorphisms of the galois groups of the maximal
  abelian extensions of imaginary quadratic fields}, Natur. Sci. Rep.
  Ochanomizu Univ. \textbf{27} (1976), no.~2, 155--161.

\bibitem{Neu2}
J\"{u}rgen Neukirch, \emph{Kennzeichnung der endlich-algebraischen
  {Z}ahlk\"{o}rper durch die {G}aloisgruppe der maximal aufl\"{o}sbaren
  {E}rweiterungen}, J. Reine Angew. Math. \textbf{238} (1969), 135--147.

\bibitem{Neu1}
\bysame, \emph{Kennzeichnung der {$p$}-adischen und der endlichen algebraischen
  {Z}ahlk\"{o}rper}, Invent. Math. \textbf{6} (1969), 296--314.

\bibitem{NeukirchANT}
J{\"u}rgen Neukirch, \emph{Algebraic number theory}, Springer Berlin,
  Heidelberg, 1999.

\bibitem{NWS}
J{\"u}rgen Neukirch, Alexander Schmidt, and Kay Wingberg, \emph{Cohomology of
  number fields, sec. ed.}, 2 ed., Springer Berlin, Heidelberg, 2013.

\bibitem{S-T}
Mohamed Sa{\"i}di and Akio Tamagawa, \emph{The m-step solvable anabelian
  geometry of number fields}, Journal für die reine und angewandte Mathematik
  (Crelles Journal) \textbf{2022} (2022), no.~789, 153--186.

\bibitem{SerreGC}
Jean-Pierre Serre, \emph{Galois cohomology}, Springer Berlin, Heidelberg, 2001.

\bibitem{UchidaIsom}
K{\^o}ji Uchida, \emph{Isomorphisms of galois groups}, Journal of The
  Mathematical Society of Japan \textbf{28} (1976), 617--620.

\bibitem{UchidaHom}
\bysame, \emph{{Homomorphisms of Galois groups of solvably closed Galois
  extensions}}, Journal of the Mathematical Society of Japan \textbf{33}
  (1981), no.~4, 595 -- 604.

\end{thebibliography}
\end{document}